\documentclass[11pt]{article}

\usepackage[left=1.0in,top=1.0in,right=1.0in,bottom=1.0in]{geometry}

\usepackage{color}

\usepackage{amsfonts,amscd,amssymb,enumerate,url}
\usepackage{amsthm}
\usepackage{amsmath} 
\usepackage{mathrsfs} 
\usepackage{bbm} %bb font numbers
\theoremstyle{plain}
\newtheorem{thm}{Theorem}[section]
\newtheorem{lemma}[thm]{Lemma}
\newtheorem{prop}[thm]{Proposition}
\newtheorem{cor}[thm]{Corollary}
\theoremstyle{definition}

\newtheorem{remark}[thm]{Remark}
\theoremstyle{example}
\newtheorem{example}{Example}

\theoremstyle{remark}
\numberwithin{equation}{section}

\newtheorem{algorithm}{Algorithm}
\newtheorem{definition}{Definition}

\usepackage{etex}
\usepackage{graphicx}
\usepackage{subfig}

\usepackage[all]{xy}

\newcommand{\ym}{y^*}

\def\T{{trap}}

\def\del{\partial}

\newcommand{\eqdef}{\, \overset{\mbox{\tiny def}}{=}\, }

\def \alp{\xi}
\def\R{\mathbb R}
\def\Z{\mathbb Z}
\def\E{\mathbb{E}}

\def\cA{\mathcal{A}}
\def\cP{\mathcal{P}}
\def\cM{\mathcal{M}}

\def\m{m}

\def\RR{\mathbb{R}} 

\newcommand{\EG}{\mathcal{A}}
\newcommand{\TG}{\mathcal{B}}

\newcommand{\X}{X^N}

\begin{document}

\title{Weak error analysis of numerical methods for stochastic models of population processes}

\author{David F. Anderson\footnote{Department of Mathematics, University of
  Wisconsin, Madison, Wi. 53706, anderson@math.wisc.edu, grant support from NSF-DMS-1009275.} \ and Masanori Koyama\footnote{Department of Mathematics, University of
  Wisconsin, Madison, Wi. 53706, koyama@math.wisc.edu, grant support from NSF-DMS-1009275 and NSF-DMS-0805793.}}

\maketitle

\begin{abstract}

The simplest, and most common, stochastic model for population processes, including those from biochemistry and cell biology, are continuous time Markov chains.  Simulation of such models is often relatively straightforward as there are easily implementable methods for the generation of exact sample paths.  However, when using ensemble averages to approximate expected values, the computational complexity can become prohibitive as the number of computations per path scales linearly with the number of jumps of the process.  When such methods become computationally intractable, approximate methods, which introduce a bias, can become advantageous.  In this paper, we provide a general framework for understanding the weak error, or bias, induced by different numerical approximation techniques in the current setting.  The analysis takes into account both the natural scalings within a given system and the step-size of the numerical method.  
Examples are provided to demonstrate the main analytical results as
well as the reduction in computational complexity achieved by the
approximate methods.

\end{abstract}


\section{Introduction}
\label{sec:intro}

%\subsection{The general problem}
This paper provides a general framework for analyzing the weak error of numerical approximation techniques for the continuous time Markov chain models typically found in the study of population processes, including chemistry and cell biology. 
The main novelty of this work lies in how the analysis takes account of both the natural multiple scalings of a given system and the step-size of the numerical method, and is best viewed as an extension of the papers \cite{AndGangKurtz2011,Li2011,Li2011b}.
 %The applicable goal of such an analysis is to be able to efficiently compute expected values of functions of the state of the models to a given tolerance.

For $k \in \{1,\dots,R\}$,  let $\zeta_k\in \R^d$ denote the possible transition directions for a continuous time Markov chain, and let $\lambda_k':\R^d \to \R$ denote the respective intensity, or propensity functions.\footnote{In the language of probability, the functions are nearly universally termed \textit{intensity} functions, whereas in the language of chemistry and cell biology these functions are nearly universally termed \textit{propensity} functions.  We choose the language of probability theory throughout the paper.}
  The  random time change representation for the model of interest is then
\begin{equation}\label{eq:main}
	X(t) = X(0) + \sum_{k = 1}^R Y_k\left( \int_0^t \lambda_k'(X(s)) ds\right) \zeta_k,
\end{equation}
where the $Y_k$ are independent, unit-rate Poisson processes.  See, for example, \cite{KurtzPop81},
\cite[Chapter 6~]{Kurtz86}, or the recent survey \cite{KurtzAnd2011}.  The infinitesimal generator for the model \eqref{eq:main} is the operator $\EG$ satisfying
\begin{equation*}
  (\EG f)(x) = \sum_{k} \lambda_k'(x)(f(x + \zeta_k) - f(x)),
\end{equation*}
where $f : \R^d \to \R$ is chosen from a sufficiently large class of functions.    
%Such a model is equivalent to the so-called ``chemical master equation'' type model commonly found in the biology literature, where, in the language of probability, the chemical master equation is Kolmogorov's forward equation.

The problem of simulating \eqref{eq:main} (and particularly of approximating
expected values) seems deceivingly easy since we can simulate the
continuous time Markov chains exactly.\footnote{This assumes the pseudo-random numbers generated by modern computers are ``random enough'' to be considered truly random.  We take this viewpoint throughout.}  Letting $f$ be some function of the state of the system giving us some quantity of interest, we may estimate $\E f(X(T))$ via an ensemble average,
\begin{equation}
	\widehat \mu_n = \frac{1}{n} \sum_{i = 1}^n f(X_{[i]}(T)),
	\label{eq:estimator1}
\end{equation}
where $X_{[i]}$ is the $i$th independent copy of \eqref{eq:main}.  The law of large numbers then ensures that 
\begin{equation}
  \lim_{n \to \infty} \widehat \mu_n = \E f(X(T)),
  \label{eq:naive}
\end{equation}  
with a probability of one.  However, it is the computational work needed to achieve an accuracy with a given tolerance, and not simply the fact that such a limit holds, that is of most interest to us.

%%%%%%%%%%%%%%%%

\subsection{Scalings and computational cost}
The main potential problem in trying to naively apply the limit \eqref{eq:naive} to a given system stems from the  fact that there is an expected computational cost to the generation of each independent realization, which we denote by $\overline N$ for now, and explicitly quantify in \eqref{eq:cost} below.  Assuming we wish to approximate $\E f(X(T))$ to an accuracy of $\epsilon>0$, in terms of confidence intervals, we must generate $O(\epsilon^{-2})$ paths yielding a total  computational complexity of order $O(\overline N \epsilon^{-2})$.  This computational complexity can be substantial when $\overline N$ is large and/or $\epsilon$ is small.

In many models of interest, including many from cell and population biology, we do, in fact, have that $\overline N \gg 1$.
% In this case, attempting to apply \eqref{eq:naive} to a desired tolerance in a reasonable amount of time may not be feasible.   
 It is therefore natural to consider how approximation schemes perform.  Before considering such schemes, however, it is important (from an analytical point of view) to modify \eqref{eq:main} by incorporating into the model a scaling parameter, $N$, that can eventually be used to quantify $\overline N$.  The value $N$ is usually taken to be the order of magnitude of $\max_i |X_i|$.  We then scale the process by setting 
 \begin{equation*}
 	X_i^N = N^{-\alpha_i}X_i,
 \end{equation*}
 where $\alpha_i$ is chosen so that $X_i^N$ is $O(1)$.  Defining $\zeta_{k}^N = N^{-\alpha_i}\zeta_{ki}$, the general form of the scaled model is then
\begin{equation}
	X^N(t) = X^N(0) + \sum_{k =1}^R Y_k\left(N^{\gamma} \int_0^t  N^{c_k} \lambda_k(X^N(s)) ds\right) \zeta_k^N,
	\label{eq:scaled}
\end{equation}
where $\gamma$ and $c_k$ are scalars such that 
\begin{equation}
  |\zeta_k^N| = O( N^{-c_k}),
  \label{eq:cond1}
\end{equation}
with $|\zeta_k^N|\approx N^{-c_k}$ for at least one $k$, and both  $X^N$ and $\lambda_k(X^N(\cdot))$ are $O(1)$.  
 We explicitly note that we are allowing for the possibility that $|\zeta_k^N| \ll N^{-c_k}$ for some of the $k$.   Also, note that the models \eqref{eq:main} and \eqref{eq:scaled} are equivalent in that one is simply a scaled version of the other.  
For concreteness, the scaling thus described will be carried out explicitly for the stochastic models arising in biochemistry in Section \ref{sec:basic_model}.

 The infinitesimal generator $\EG^N$ for the model \eqref{eq:scaled} is 
\begin{equation}  \label{gen0}
  (\EG^N f)(x) = N^{\gamma} \sum_{k} N^{c_k} \lambda_k(x)(f(x + \zeta_k^N) - f(x)),
\end{equation}
where $f : \R^d \to \R$ is chosen from a sufficiently large class of functions.    
  Note that it is now natural to take
 \begin{equation}
 	\overline N = N^{\gamma} \sum_k N^{c_k}
	\label{eq:cost}
 \end{equation}
 as the order of magnitude for the number of steps required to generate a single path up to a time of $T>0$.  

The parameter $\gamma$ of \eqref{eq:scaled} should be thought of as representing the natural time-scale of the problem with $\gamma>0$ implying a relevant time scale smaller than one.  In this case of $\gamma>0$, the explicit numerical schemes considered in this paper are usually \textit{not} a good choice and other methods, such as averaging techniques, are usually required in conjunction with the methods described here \cite{KurtzAnd2011,Ball06,CaoSS,ELiuVE2006,KurtzKang}.  This fact is demonstrated by our main analytical results which provide error bounds for the different schemes that grow exponentially in $N^{\gamma}$.  Therefore, our main results are most useful when $\gamma \le 0$.

\vspace{.125in}

 The model \eqref{eq:scaled} is henceforth our main model of interest.  We make the following running assumption, which in light of the fact that both $X^N$ and $\lambda_k(X^N(\cdot))$ are $O(1)$, is a light one.
\vspace{.2in}

\noindent \textbf{Running Assumption:} The intensity functions $\lambda_k$ for the scaled process $X^N$ satisfying \eqref{eq:scaled}, together with all of their derivatives, are uniformly bounded.  
\vspace{.2in}

The above running assumption can almost certainly be weakened to a local Lipschitz condition, in which case analytical methods similar to those found in \cite{Hutzenthaler2011} and/or \cite{MattinglyStuart} can be applied.  Proving our main results in such generality, while possible and certainly worth doing in future work, will be significantly messier and we feel the main points of the analysis will be lost.

\vspace{.1in}
We return to our problem of interest and let $f$ be a function of the state giving a quantity of interest and consider how to approximate $\E f(X^N(T))$.    %Note that we are now approximating an expected value of the \textit{scaled} process, $X^N$, even though the  quantity of interest may be a statistic of the \textit{unscaled} process $X$.  The reason for this is two-fold.  First, working with the scaled process \eqref{eq:scaled} is more analytically tractable than working with \eqref{eq:main} as it explicitly takes the $N$ dependence, which is presumably large, into account in the analysis.  Second, for the most common choices of $f$, such as 
%\begin{enumerate}
%\item $f(x) \equiv x$, 
%\item $f(x) \equiv x^2$, 
%\item $f(x) \equiv 1_{\{x \in A\}}$, where $A \subset \R^d$,
%\end{enumerate}
%it is usually simple to convert from $f(X^N)$ to $f(X)$.
As already discussed, the computational cost of approximating $\E f(X^N(T))$ to an accuracy of $\epsilon$, in the sense of confidence intervals, using the estimator \eqref{eq:estimator1} is $O(\overline N \epsilon^{-2})$.  Suppose now that $Z^N$ is an approximation of $X^N$ constructed with a time-discretization step of size $h>0$.\footnote{As the approximate process explicitly depends upon the choice of $h$, we could denote it as $Z^N_h$.  However, for  ease of exposition we choose to drop the $h$ dependence from the notation.} Letting $Z_{[i]}^N$ denote independent copies of $Z^N$, we construct the estimator 
\begin{equation}
  \widehat \mu_n = \frac{1}{n}\sum_{i = 1}^n f(Z_{[i]}^N(T)).
  \label{eq:CMC}
\end{equation}
Suppose that it can be shown that the approximation scheme has a weak error, or bias, of order one.  That is,
\begin{equation*}
	\E f(X^N(T)) - \E f(Z^N(T)) = O(h),
\end{equation*}
for a suitably large class of functions $f$.  Then, noting that
 \begin{equation*}
 	\E f(X^N(T)) - \widehat \mu_n = \left[ \E f(X^N(T)) - \E f(Z^N(T))\right] + \left[ \E f(Z^N(T)) - \widehat \mu_n \right],
 \end{equation*}
 we see that we must choose $h = O(\epsilon)$ to make the first term on the right, the bias, $O(\epsilon)$, and $n = O(\epsilon^{-2})$ to make the second term, the statistical error, have a variance of $O(\epsilon^2)$, and a standard deviation of $O(\epsilon)$.  This gives a total computational complexity of $O(\epsilon^{-3})$.  This will greatly lower the computational complexity of the problem, as compared with using  exact sample paths, if $\epsilon^{-1} \ll \overline N$.  
 
 If, instead, the method for generating $Z^N$ is second order accurate in a weak sense, that is if
 \begin{equation*}
	\E f(X^N(T)) - \E f(Z^N(T)) = O(h^2),
\end{equation*}
then we could choose $h = O(\epsilon^{1/2})$  to yield a bias of $O(\epsilon)$.  This leads to a total computational complexity of $O(\epsilon^{-2.5})$, which for small $\epsilon$ represents a substantial improvement over using an order one method.

 The above discussion points out that the key quantity to understand for a given approximation method, and the focus of this paper, is the bias, or weak error, it induces for a given function $f$:
\begin{equation}
  B_f(Z^N,x,t) \eqdef  \E_x f(X^N(t)) - \E_x f(Z^N(t)).
\label{eq:bias}
\end{equation}
Note that $B_f(Z^N,x,h)$ represents the local, one-step error of the method as the fixed time-step is of size $h>0$.
Analyzing the bias induced by different numerical schemes is by now classical in the study of stochastic processes, with nearly all the focus falling on how the bias scales with the size of the time-step, $h$ \cite{KloedenPlaten92}. However, it is not sufficient in the current setting to simply understand how the bias \eqref{eq:bias} scales with the time-discretization alone.  Care must also be taken to quantify how the leading order constants depend upon the natural scalings of a given system and given method, here quantified by the parameter $N>0$.   For example, if
 \begin{equation*}
	\E f(X^N(T)) - \E f(Z^N(T)) = O( c_1^Nh + c_2^Nh^2),
\end{equation*}
then we wish to understand how $c_1^N, c_2^N$ depend upon $N$ since for a given choice of $h$ we may have that $c_1^Nh < c_2^N h^2$.  In this case,  the method will behave as if it is an order two method until $h$ is reduced to the point when $c_1^Nh > c_2^N h^2$, in which case it will behave like an order one method.

\subsection{Notation and terminology}
\label{sec:terminology}

 In this short subsection, we collect some necessary extra notation and terminology used throughout the paper. We first note that for  $f : \R^d \to \RR$, and any $t\ge0$, Dynkin's formula for the process \eqref{eq:scaled} is 
\begin{equation}\label{eq:Dynkin}
	\E_{x} f(X^N(t)) = f(x) + \E_{x} \int_0^t \EG^N f(X(s)) ds,
\end{equation}
which holds so long as the expectations exist.  Similar expressions will hold for the approximate methods under consideration.  Dynkin's formula will be our main analytical tool as it will allow us to quantify the bias \eqref{eq:bias} for the different methods and we therefore focus on developing compact notation for the generators of our processes.

We define the operator $\nabla_k^N$  for the $k$th possible transition, which we will typically call a ``reaction'' in keeping with the motivating application of Section \ref{sec:basic_model}, via
\begin{equation} 
  \nabla^N_k f(x) \eqdef  N^{c_k}(f(x + \zeta_k^{N})   - f(x)). 
  \label{delta} 
  \end{equation}
%    For example, in Example \ref{ex:gamma} above, we have
%  \begin{align*}
%  	\nabla_1^N f(x) &= N(f(x + (e_2 - e_1)/N) - f(x))\\
%	\nabla_2^N f(x) &= N(f(x + (e_1 - e_2)/N) - f(x)),
%  \end{align*}
%  where $S_1 \to S_2$ is arbitrarily labeled as the first reaction, and $e_i\in \Z^2$ is the vector of all zeros except with a one in the $i$th location.  
     Note that if $f$ is globally Lipschitz, then $\nabla^N_k f(x)$ is uniformly bounded over $k$ and $x$ since  $|\zeta_k^N| = O(N^{-c_k})$.  
    We may now write \eqref{gen0} as
\begin{equation*}
  (\cA^N f)(x) = \sum_{k} N^\gamma \lambda_k(x)  \nabla^N_k f(x) .
  \end{equation*}
Defining the vector valued operators 
\begin{equation}\label{eq:operators}
   \lambda \eqdef [\lambda_1, \dots , \lambda_R] , \quad   \nabla^N \eqdef [\nabla_1^N,\dots , \nabla_R^N], 
\end{equation}
where we recall that $R$ is  the number of reactions,  we obtain  
\begin{align*}
(\cA^N f) (x) = ( N^\gamma \lambda \cdot  \nabla^N ) f(x).
 %\label{gen}
\end{align*}
For $i \in \{1,\dots,d\}$ and $k \in \{1,\dots,R\}$, we let $m_k$ satisfy
\begin{equation*}
  |\zeta_k^N| = N^{-\m_k}.
\end{equation*}
Note that, by construction, we have $c_k \le \m_k$, for all $k$.  
Finally,  we denote the $j$th directional derivative of $f$ into the direction $[v_1, v_2, ... v_j ]$ by $f'[v_1, ..., v_j]$ and make the usual definition
\begin{align}
\| f \|_j   \eqdef \sup_{x}   \{  f'[v_1, ...., v_j](x),  \|v\| = 1    \}  
\label{eq:norm}
\end{align}

%%%%%%%%%%%%%%%%%%%%%%%%%%%%%%

\subsection{Summary of main results.}  
The following list is a summary of our main results.  Technical details and assumptions have been omitted from the statements below for the sake of clarity.  
%In the following subsection, we discuss how these results are related to previous results in \cite{AndGangKurtz2011, AndMatt2011, Li2011, Li2011b}.
\begin{enumerate}[1.]
\item In Theorem \ref{thm:LG}, we prove that for any explicit numerical scheme with a step-size of $h>0$, 
\begin{equation*}
	B_f(Z^N,x,T) = O(Th^{-1}\sup_{z}|B_f(Z^N,z,h)|),
\end{equation*}
where $B_f(Z^N,x,t)$ is the bias  defined in \eqref{eq:bias}.  Thus,  if the numerical scheme has a local, one-step, error of $O(h^{p+1})$, then the global error is $O(h^{p})$.    This result is standard and should be compared to similar results  in \cite{AndMatt2011,Li2011}.  It is included here since it is necessary to show that  the scalings do not alter the usual result.

\item In Theorem \ref{thm:local_Euler}, we prove that if $Z_E^N$ is generated via Euler's method, also known as explicit $\tau$-leaping in the setting of biochemistry, then
\begin{equation*}
	B_f(Z_E^N,x,h) = O( c h^2),
\end{equation*}
where $c$ is independent of $N$.  Thus, after applying Theorem \ref{thm:LG}, Euler's method is proven to be an order one method in that the leading order of the global error satisfies
\begin{equation*}
	B_f(Z^N_E,x,T) = O(ch),
\end{equation*}
and  decreases linearly with the step-size.  This fact is formally stated in Theorem \ref{thm:euler_global}.

\item In Theorem \ref{thm:mdpt_local}, we prove that if $Z_M^N$ is generated via an approximate midpoint method, then
\begin{equation*}
	B_f(Z^N_M,x,h) = O( c_1^N h^2 + c_2^N h^3),
\end{equation*}
where $c_1^N,c_2^N$ depend upon the natural scalings of the system, quantified here by $N> 0$.   The term that dominates this error  then depends upon the specific scalings of a system, encapsulated in the constants $c_1^N$ and $c_2^N$,  and the size of the time discretization $h$.  Theorem \ref{thm:LG} then implies
\begin{equation*}
	B_f(Z^N_M,x,T) = O( c_1^N h + c_2^N h^2),
\end{equation*}
and the midpoint method will sometimes behave like a first order method, and other times will behave like a second order method.  This fact is formally stated in Theorem \ref{thm:mdpt_global}. A transition point, in terms of $N$ and $h$, for this change in behavior is also provided.

\item In Theorem \ref{thm:local_trap}, we prove that if $Z^N_{\T}$ is generated via the weak trapezoidal method, which was originally formulated in the diffusive setting  \cite{AndMatt2011} and is extended to the discrete setting in Section \ref{sec:numerics}, then
\begin{equation*}
	B_f(Z^N_{\T},x,h) = O( c h^3),
\end{equation*}
where $c$ is independent of $N$.  Thus, after applying Theorem \ref{thm:LG}, the Weak Trapezoidal method is proven to be a second order method in that the leading order of the global error satisfies
\begin{equation*}
	B_f(Z^N_{\T},x,T) = O(ch^2),
\end{equation*}
and  decreases quadratically with the step-size.  This fact is formally stated in Theorem \ref{thm:wtm_global}.  %To the best of the authors' knowledge, this constitutes the first higher order method in the current setting.

%%%%%%%%%%%%%%%%%%%%%%%%%%%%%%%%%%%%
%%%%%%%%%%%%%%%%%%%%%%%%%%%%%%%%%%%%

\subsection{Context}

We attempt to put the present work in the correct historical context.  
  In \cite{AndGangKurtz2011}, Anderson, Ganguly, and Kurtz provided the first error analysis of different approximation techniques that incorporated the natural scalings of the system \eqref{eq:main} into the analysis.  Specifically, they considered models satisfying the ``classical scaling,'' which using our present terminology corresponds with $\gamma = 0$, $c_k \equiv 1$, and $\alpha_i \equiv 1$.  They further coupled the time discretization to the scaling, thereby ensuring $h$ was always in a useful regime, and derived results for both the weak and strong error of Euler's method and the midpoint method.
  They proved that, in this specific setting, Euler's method is an order one method in both a weak and a strong (in the $L^1$ norm) sense.  They proved that the strong error of the midpoint method falls between order one and two (see \cite{AndGangKurtz2011} for precise statements), and that the leading order term of the weak error of the midpoint method scales quadratically with the step-size.

%The importance of the analysis in \cite{AndGangKurtz2011} is that it
% pointed out the need to incorporate the natural scales of the system into the analysis.  Our main results pertaining to the midpoint method, Theorems \ref{thm:mdpt_local} and \ref{thm:mdpt_global}, agree with, though say substantially more than, the results found in \cite{AndGangKurtz2011}.

In \cite{Li2011} it was shown by Hu, Li, and Min that the $O(h^2)$ weak convergence rate of  the midpoint method given in \cite{AndGangKurtz2011} depended intimately on the coupling between the time discretization and the scaling parameter of the system.  It was this particular observation that in a large part motivates the present work as we wish to provide a general analytical rate of convergence for the different relevant methods in the most general possible scaling regime and in which the time discretization parameter is  independent  of the natural scalings.

In \cite{AndMatt2011} the weak trapezoidal algorithm was introduced in the context of SDEs driven by Brownian motions.  There, it was shown to be an easy to implement method that is second order accurate in a weak sense.  Further, it has the nice property that no costly derivatives need be computed during the course of the simulation.  In \cite{Li2011b}, it was pointed out that since the motivation for the original weak trapezoidal algorithm comes from viewing SDEs as driven by space-time Wiener processes, the exact same algorithm can work in the present jump setting by viewing the driving forces as space-time Poisson processes.  This observation was also made independently in an earlier version of the present paper.

It is worth pointing out that there are at least two other trapezoidal type algorithms in the literature pertaining to models of stochastic chemical kinetics.  These are the implicit and explicit trapezoidal methods of \cite{CaoTrap2005}.  These method were explicitly developed to give better stability than the usual methods, and so do not exhibit better convergence than does Euler's method.

 \end{enumerate}

%%%%%%%%%%%%%%%%%%%%%

\subsection{Paper outline}
  
The  remainder of the paper is organized as follows.  In Section \ref{sec:basic_model}, we show how the basic models considered in this paper, both \eqref{eq:main} and \eqref{eq:scaled}, arise naturally in biochemistry, which is the main area of motivation for this work.  This section can safely be skipped by anyone not interested in that application.   In Section \ref{sec:numerics}, we discuss numerical methods for the models under consideration, including both exact and approximate schemes. 
% In Section \ref{sec:scaling}, we discuss the multi-scale nature of stochastic models of chemical reaction networks, and incorporate this into our model.   
In Section \ref{sec:analysis}, we prove Theorem \ref{thm:LG}, as stated loosely above, and relevant corollaries.  In Section \ref{sec:local}, we prove Theorems \ref{thm:local_Euler}, \ref{thm:mdpt_local}, and \ref{thm:local_trap}, each stated loosely above, providing the local, one-step errors induced by the approximate schemes considered here.  In Section \ref{sec:regularity}, we provide bounds on the semigroup operator of the exact process $X^N$, yielding the final piece to the global analysis of the weak error of the different methods.  We also briefly discuss stability concerns in Section \ref{sec:regularity}.  In Section \ref{sec:examples}, we provide relevant examples.

\section{Motivating Systems: Biochemical Reaction Networks}
\label{sec:basic_model}

This section builds the relevent models \eqref{eq:main} and \eqref{eq:scaled} used in the study of stochastically modeled biochemical reaction networks.  We feel it is worthwhile to include this section as this is the area of main motivation for the present work.  However, it can safely be skipped by those wishing to simply see the mathematical analysis and not the areas of application.

\subsection{The unscaled model}
\label{sec:first_biochem}

A chemical reaction network is a dynamical system involving multiple reactions and chemical species. The simplest stochastic models of such networks treat the system as a
 continuous time Markov chain with the state, $X\in \Z^d_{\ge 0}$, giving the number of molecules of each species and with reactions modeled as possible transitions of the chain. 

An example of a chemical reaction is 
\begin{align*}
  2S_{1}+S_{2} ~\rightarrow~ S_{3},
\end{align*}
where we would interpret the above as saying two molecules of type $S_1$ combine with a molecule of type $S_2$ to produce a molecule of type $S_3$. The $S_{i}$ are called chemical {\em species}. Letting 
\begin{equation*}
  \nu_1 = \left(\begin{array}{c}
  2\\
  1\\
  0
  \end{array}\right), \quad \nu_1' = \left(\begin{array}{c}
  0\\
  0\\
  1
  \end{array}\right), \quad \text{and} \quad \zeta_1 = \nu_1' - \nu_1 = \left(\begin{array}{c}
  -2\\
  -1\\
  1
  \end{array}\right),
  \end{equation*}
   we see that every instance of the reaction changes the state of the system by addition of $\zeta_1$.  Here the subscript ``1'' is used to denote the first (and in this case only) reaction of the system.

 In the
general setting we denote the number of species by $d$, and for $i \in \{1,\dots, d\}$ we denote the $i$th species as $S_{i}$.  We then
consider a finite set of $R$ reactions, where the model for the $k$th reaction is
determined by
\begin{enumerate}[$(i)$]
\item a vector of inputs $\nu_k$ specifying the number of
molecules of each chemical species that are consumed in the reaction,
\item a vector of outputs $\nu_k'$ specifying the number of molecules of
each chemical  species that are created in the reaction, and 
\item a function of the
state $\lambda_k'$ that gives the transition intensity, rate, or propensity
at which the reaction occurs.
\end{enumerate}
Specifically, if we denote the state of the system at time $t$ by $X(t) \in \Z^d$, and if the $k$th reaction occurs at time $t$, we update the state by addition of the \textit{reaction vector} 
\begin{equation*}
  \xi_k \eqdef \nu_k' - \nu_k
  \end{equation*}
   and the new state becomes $\displaystyle X(t) = X(t-) + \xi_k.$ For the standard Markov chain model, the number of times that the $k$th reaction occurs by time $t$ can be
represented by the counting process 
$$ R_k(t) = Y_k\bigg(\int_0^t \lambda_k'(X(s))ds\bigg), $$ 
where the $Y_k$ are independent,
unit-rate Poisson processes \cite{KurtzPop81}, \cite[Chapter 6~]{Kurtz86}.  The state of the system then satisfies
\begin{align*}
  X(t) &= X(0) + \sum_k Y_k\left( \int_0^t \lambda_k'(X(s))ds
  \right)\xi_k,
\end{align*}
which was \eqref{eq:main} in the Introduction.
The above formulation is termed a random time change representation
and is equivalent to the chemical master equation representation
found in much of the biology and chemistry literature, where the master
equation is Kolmogorov's forward equation in the terminology of
probability.

A common choice of intensity function for chemical reaction systems is that of mass action kinetics.  Under mass
action kinetics, the intensity function for the $k$th reaction is
\begin{equation}
\lambda_k'(x) = \kappa_k' \prod_{i = 1}^d \frac{x_i!}{(x_i - \nu_{ki})!},% \kappa_k \left(\prod_{i=1}^d \nu_{ik}! \right) \prod_{i=1}^d {{x_i}\choose {\nu_{ik}}}.
  \label{eq:stoch_MA}
\end{equation}
where $\nu_{ki}$ is the $i$th component of $\nu_k$.
% Implicit in the assumption of mass action kinetics is that the vessel under consideration is ``well-stirred.'' 

\begin{example}
  To solidify notation, we consider the network
  \begin{equation*}
  	S_1 \overset{\kappa_1}{\underset{\kappa_2}{\rightleftarrows}} S_2,\qquad 2S_2 \overset{\kappa_k}{\rightarrow} S_3,
  \end{equation*}
  where we have placed the rate constants $\kappa_k$ above or below their respective reactions.  For this example, equation \eqref{eq:main} is
  \begin{align*}
  	X(t) = X(0) &+ Y_1\left( \int_0^t \kappa_1 X_1(s)ds\right)\left[\begin{array}{c}
	-1\\
	1\\
	0
	\end{array} \right] + Y_2\left( \int_0^t \kappa_2 X_2(s)ds\right)\left[\begin{array}{c}
	1\\
	-1\\
	0
	\end{array} \right] \\
	&+ Y_3\left( \int_0^t \kappa_3 X_2(s)(X_2(s)-1)ds\right)\left[\begin{array}{c}
	0\\
	-2\\
	1
	\end{array} \right].
  \end{align*}
  Defining $\zeta_1 = [-1,1,0]^T$, $\zeta_2 = [1,-1,0]^T$, and $\zeta_3 = [0,-2,1]^T$, the generator $\cA$ satisfies
  \begin{equation*}
  	(\cA f)(x) = \kappa_1 x_1(f(x +\zeta_1 ) - f(x)) + \kappa_2 x_2(f(x + \zeta_2) - f(x))+ \kappa_3 x_2(x_2 - 1)(f(x + \zeta_3) - f(x)).
  \end{equation*}
\end{example}

\subsection{Scaled biochemical models}
\label{sec:scaling} 

%As discussed in and around \eqref{eq:RateAssumption},  the approximate algorithms being considered are only useful on the class of models which satisfy $\sum_k \lambda_k(X(\cdot)) \gg 1$.  There are at least two different ways this behavior can be achieved.  The first is that there could be a large number of reactions, $R \gg 1$, in which case the approximate algorithms currently being discussed will not provide an appreciable improvement in terms of runtime over the exact simulation methods.  The other common way for $\sum_k \lambda_k(X(\cdot)) \gg 1$ to hold is to have either large abundances of certain species, or to have large rate constants, or both.  We will study the behavior of the different algorithms under this latter assumption.  To do so, we will introduce a scaling parameter, $N$, used to quantify the  variations in the sizes of the abundances and parameters.   

The scaling described below has been used previously in at least \cite{AndersonHigham2011, KurtzAnd2011, Ball06, KurtzKang}.  We emphasize that the scaling is an analytical tool used to understand the behavior of the different processes, and that  the actual simulations using the different methods make no use of, nor have need for, an understanding of $N$, $\alpha$, or the $\beta_k$. 

Let $N \gg 1$ be a natural parameter of the system, perhaps the abundance of the species with the highest number of molecules.  Assume that the system satisfies \eqref{eq:main} with $\lambda_k'$ determined via mass-action kinetics \eqref{eq:stoch_MA}, and  $\zeta_k\in \Z^d$ representing the reaction vectors described in Section \ref{sec:first_biochem}.  For each species $i$, define the {\em normalized  abundance\/}
 (or simply, the abundance) by
\begin{equation*}
  \X_i(t)=N^{-\alpha_i}X_i(t), \label{normalized}
\end{equation*}
where $\alpha_i\geq 0$ should be selected so that $\X_i$ is $O(1)$. 
Here $X_i^N$  may be the species number ($\alpha_i=0$) or the
species concentration, or something else.

Since the rate constants may also vary over several orders of
magnitude, we write $\kappa_k'=\kappa_kN^{\beta_k}$ where the
$\beta_k$ are selected so that $ \kappa_k=O(1)$ (recall that $\kappa_k'$ is the original system parameter).  Note that while the $\alpha_i$ are non-negative if $N$ is chosen to be the abundance of the species with the highest number of molecules,  $\beta_k$ can be positive, negative, or zero.
%
%Note that for a binary
%reaction
%\begin{equation*}
%  \kappa_k'X_i X_j=N^{\beta_k+\alpha_i+\alpha_j}\kappa_kX^N_i X^N_j,
%\end{equation*}
%and we can write
%\begin{equation*}
%  \beta_k+\alpha_i+\alpha_j=\beta_k+\nu_k \cdot \alpha.
%\end{equation*}
%We also have,
%\begin{equation*}
%  \kappa_k'X_i=N^{\beta_k+\nu_k \cdot \alpha}\kappa_k X^N_i,\quad\kappa_k'X^N_
%  i(X^N_i-1)=N^{\beta_k+\nu_k \cdot \alpha} \kappa_k X^N_i(X^N_i-N^{-\alpha_i}),
%\end{equation*}
%where the source vectors are $\nu_k = e_i$ in the first example and $\nu_k = 2e_i$ in the second, with similar expressions for intensities involving higher order
%reactions.  That is, 

Under the mass-action kinetics assumption, we always have that $\lambda_k'(X(s))=N^{\beta_k + \nu_k\cdot \alpha}\lambda_k(X^N(s))$, where $\lambda_k$ is deterministic mass-action kinetics with rate constants $\kappa_k$ \cite{KurtzAnd2011, Ball06, KurtzKang}.  
%Note that for reactions of
%the form $2S_i\to *$, where $*$ represents an arbitrary linear combination of the species, the rate is
%$N^{\beta_k+2\alpha_i}\kappa_k \X_i(t)(\X_i(t)-N^{-\alpha_i})$, so if
%$\alpha_ i>0$, we should write $\lambda_k^N$ instead of $\lambda_k$,
%but to simplify notation, we will simply write $\lambda_k$. 
Our model has therefore become 
\begin{equation}
  \X(t)= \X(0)+\sum_k Y_k\left(\int_0^tN^{\beta_k+\nu_
    k\cdot\alpha} \lambda_k(\X(s))ds\right)\zeta_{k}^N, \quad i \in \{1,\dots, d\},
    \label{eq:main_multi}
\end{equation}
where $\zeta_{ki}^N \eqdef N^{-\alpha_i} \zeta_{ki}$.  To quantify the natural time-scale of the system,  define $\gamma\in \R$ via
\begin{align*}
\begin{split}
\gamma = \max_{\{i,k \ : \ \zeta_{ki}^N \ne 0\} } \{ \beta_k + \nu_k \cdot \alpha - \alpha_i\},
\end{split} \label{eq:gamma}
\end{align*}
where we recall that $\nu_k$ is the source vector for the $k$th reaction.  Letting
\begin{equation*}
	c_k = \beta_k + \nu_k\cdot \alpha - \gamma,
\end{equation*}
for each $k$, the model \eqref{eq:main_multi} is seen to be exactly \eqref{eq:scaled}.

\begin{remark}\label{remark:classical}
If $\beta_k + \nu_k\cdot \alpha = \alpha_i = 1$ for all $i,k$ in \eqref{eq:main_multi}, in which case $\gamma = 0$, then we have what is typically called the \textit{classical scaling}.  It was specifically this scaling that was used in the analyses of the Euler and midpoint methods found in \cite{AndGangKurtz2011, Li2011,Li2011b}. In this case it is natural to consider $\X$ as a vector whose $i$th component gives the concentration, in moles per unit volume, of the $i$th species.
\end{remark}

\begin{example}\label{ex:gamma}
As an instructive example, consider the system
\begin{equation*}
	S_1 \overset{100}{\underset{100}{\rightleftarrows}} S_2
\end{equation*}
with $X_1(0) = X_{2}(0) = \text{10,000}$.  In this case, it is natural to take $N = $ 10,000 and $\alpha_1 = \alpha_2 = 1$.  As the rate constants are $100 = \sqrt{\text{10,000}}$, we take $\beta_1 = \beta_2 = 1/2$ and find that $\gamma = 1/2$.  The equation governing the normalized process $X^N_1$ is
\begin{equation*}
    X_1^N(t) = X_1^N(0) - Y_1\bigg(N^{1/2}N \int_0^t X_1^N(s)ds\bigg)\frac{1}{N} + Y_2\bigg(N^{1/2}N \int_0^t (2 - X_1^N(s))ds\bigg)\frac{1}{N}
\end{equation*}
where we have used that $X^N_1 + X^N_2 \equiv 2$.
\end{example}

%%%%%%%%%%%%%%%%%%%%%%%%%%%%%%%%%%%%
%%%%%%%%%%%%%%%%%%%%%%%%%%%%%%%%%%%%

\section{Numerical methods}
\label{sec:numerics}

\subsection{Exact methods.}  
\label{sec:exact_methods}

As already discussed in the introduction, because we are considering continuous time Markov chains,  there are a number of numerical methods available for the  generation of exact sample paths for the model \eqref{eq:main}, or the equivalent model \eqref{eq:scaled}.   All are examples of discrete event simulation \cite{Glynn89}.  In the language of biochemistry these methods include the
stochastic simulation algorithm, best known as Gillespie's algorithm in this setting
\cite{Gill76,Gill77}, the first reaction method \cite{Gill76}, and
the next reaction method \cite{Anderson2007a, Gibson2000}.  All such
algorithms perform the same two basic steps multiple times until a
sample path is produced over a desired time interval:  conditioned on the current state of the system, both $(i)$ the amount of time that
passes until the next reaction takes place, $\Delta t$, is computed and $(ii)$ the specific reaction that has taken place is found.  Note that $\Delta t$ is an exponential random variable with a parameter of $\sum_k \lambda_k(X(t))$.   Therefore, if
\begin{equation}\label{eq:RateAssumption}
\sum_k \lambda_k(X(t)) \approx \overline N \gg 1 \qquad \text{so that} \qquad \E\Delta t = \frac{1}{\sum_k
\lambda_k(X(t))} \approx \frac{1}{\overline N} \ll 1,
\end{equation}
then the runtime needed to produce a single
exact sample path may be prohibitive when coupled with Monte Carlo techniques, and approximate methods may be desirable.

\subsection{Approximate methods.}
\label{sec:approximate_methods}
%Throughout the paper, we let $X$ denote the solution to \eqref{eq:RTC_exact}.  For approximate methods, w
There will be times we will wish to discuss an arbitrary approximation to $X$ or $X^N$, and other times  we will wish to consider specific approximations.  When we consider an arbitrary approximation we will simply denote the approximation as $Z$ or $Z^N$.  When we distinguish the Euler, midpoint, and Weak Trapezoidal approximations, the main approximations under consideration here, we will denote by $Z_E$, $Z_M$, and $Z_\T$ the respective approximations to $X$, and by $Z_E^N$, $Z_M^N$, and $Z_\T^N$ the respective approximations to $X^N$.  Throughout, our time-discretization parameter will be denoted by $h>0$.\footnote{Historically, the time discretization parameter for the methods described in this paper have been $\tau$, thus giving these methods the general name ``$\tau$-leaping methods.''  We choose to break from this tradition and denote our time-step by $h$ so as not to confuse $\tau$ with a stopping time.}

\subsubsection{Euler's method}
The Euler approximation, $Z_E$, to the model \eqref{eq:main} is the solution to 
\begin{equation}
  Z_E(t) = Z_E(0) + \sum_k Y_k \left( \int_0^t \lambda_k'(Z_E \circ \eta(s))
    ds  \right)\zeta_k, 
  \label{eq:Euler}
\end{equation}
where $\displaystyle \eta(s) \eqdef \left \lfloor \frac{s}{h} \right \rfloor h$, and all other notation is as before. Note that $Z_E(\eta(s)) = Z_{E}(t_n)$ if $t_n\le s < t_{n+1}$.  The basic algorithm for the simulation of \eqref{eq:Euler} up to a time of $T>0$ is the following.   For $x \ge 0$ we will write Poisson$(x)$ for a Poisson random variable with a parameter of $x$.

\begin{algorithm}[Euler's method] Fix $h>0$.  Set $Z_E (0) = x_0$, $t_0 = 0$,
  $n=0$ and repeat the following until $t_{n+1} = T$:
  \begin{enumerate}[$(i)$]
   \item Set $t_{n+1} = t_{n} + h$.  If $t_{n+1} \ge T$, set $t_{n+1} = T$ and $h = T - t_n$.
  \item For $k \in \{1,\dots,R\}$, let $\Lambda_k = \text{Poisson}(\lambda_k'(Z_E(t_n))h)$ be independent of each other and all previous random variables.
  \item Set $Z_E(t_{n+1}) = Z_E(t_{n}) + \sum_k \Lambda_k \zeta_k$.
   \item Set $n \leftarrow n+1$.
  \end{enumerate}
  \label{alg:Euler}
\end{algorithm}

The above algorithm is termed  \textit{explicit tau-leaping} in the biology and biochemistry literature \cite{Gill2001}.  Several improvements and modifications have been made to the basic
algorithm described above over the years in the context of biochemical processes.  Many of the improvements are concerned with how to choose the step-size adaptively \cite{Cao2006,
  Gill2003} and/or how to ensure that population values do not go
negative during the course of a simulation \cite{Anderson2007b,
  Cao2005, Chatterjee2005}, which is a relevant issue as population processes have a natural non-negativity constraint.  For the simulations carried out in Section \ref{sec:examples}, we choose to simply keep a fixed step-size and set any species that goes negative in the course of a jump to zero.

Defining the operator
\begin{equation}
  (\TG_z f)(x) \eqdef \sum_k \lambda_k'(z)(f(x + \zeta_k) - f(x)), 
  \label{eq:tau_gen}
\end{equation}
we see that for $t>0$ 
\begin{equation}
  \E f(Z_E(t)) = \E f(Z_E\circ \eta(t)) + \E \int_{\eta(t)}^t (\TG_{Z_E\circ
    \eta(t)} f)(Z_E(s)) ds, 
  \label{eq:tau_martingale}
\end{equation}
so long as the expectations exist.    
%Equation \eqref{eq:tau_martingale} points out why we care about the associated operators for each of our approximate methods:   they will be used to gain the necessary control over the difference $\E f(X(t)) - \E f(Z(t))$, called the \textit{weak error} of the approximation, which is the focus of our paper.
The scaled version of \eqref{eq:Euler}, which is an approximation to $X^N$ satisfying \eqref{eq:scaled}, is
\begin{equation}
   Z^N_E(t) = Z^N_E(0) + \sum_k Y_k \left( N^{\gamma} \int_0^t  N^{c_k} \lambda_k(Z_E^N \circ \eta(s))
    ds  \right)\zeta_k^N, 
    \label{eq:RTC_tau_scaled}
\end{equation}
where all notation is as before.
 Define the operator $\TG_z^N$ by
\begin{equation}
	\TG_{z}^N f(x) \eqdef (N^{\gamma} \lambda(z) \cdot \nabla^N) f(x).
	\label{eq:scaled_operator}
\end{equation}
If $Z^N_E$ satisfies \eqref{eq:RTC_tau_scaled},  then for all $t>0$
\begin{equation*}
  \E f(Z_E^N(t)) = \E f(Z_E^N ( \eta(t))) + \E \int_{\eta(t)}^t
  (\TG_{ Z_E^N ( \eta(t))}^N f)(Z_E^N(s)) ds,
\end{equation*}
so long as the expectations exist.

\subsubsection{Approximate midpoint method}

A midpoint type method was first described in \cite{Gill2001}\footnote{The midpoint method detailed in \cite{Gill2001} is actually a slight variant of the method described here.  In \cite{Gill2001} the approximate midpoint, called $\rho(z)$ above, is rounded to the nearest integer value.} and analyzed in \cite{AndGangKurtz2011}  and \cite{Li2011}.  Define the function
\begin{equation*}
  \rho(z) \eqdef z + \frac{1}{2}h\sum_k \lambda_k'(z) \zeta_k,
\end{equation*}
which computes an approximate midpoint for the system \eqref{eq:main} assuming the
state of the system is $z$ and the time-step is $h$.  Then define $Z_M$ to be the process that satisfies
\begin{equation}
  Z_M(t) = Z_M(0) + \sum_k Y_k \left( \int_0^t \lambda_k' \circ \rho
    ( Z_M \circ  \eta (s) ) ds  \right)\zeta_k. 
    \label{eq:pathwise_mdpt}
\end{equation}

The basic algorithm for the simulation of \eqref{eq:pathwise_mdpt} up to a time of $T>0$ is the following.   Note that only step $(ii)$ changes from Euler's method.

\begin{algorithm}[Midpoint method] Fix $h>0$.  Set $Z_M(0) = x_0$, $t_0 =
  0$, $n=0$ and repeat the following until $t_{n+1} = T$:
  \begin{enumerate}[$(i)$]
     \item Set $t_{n+1} = t_{n} + h$.  If $t_{n+1} \ge T$, set $t_{n+1} = T$ and $h = T - t_n$.
  \item For $k \in \{1,\dots,R\}$, let $\Lambda_k = \text{Poisson}(\lambda_k'\circ \rho(Z_M(t_n))h)$ be independent of each other and all previous random variables.
  \item Set $Z_M(t_{n+1}) = Z_M(t_{n}) + \sum_k \Lambda_k \zeta_k$.
   \item Set $n \leftarrow n+1$.
   \end{enumerate}
  \label{alg:mdpt}
\end{algorithm}

 For $\TG_z$ defined via
\eqref{eq:tau_gen}, any $t>0$, and $Z_M$ satisfying \eqref{eq:pathwise_mdpt}, we have
\begin{equation*}
  \E f(Z_M(t)) = \E f(Z_M \circ \eta(t)) + \E \int_{\eta(t)}^t
  (\TG_{\rho \circ Z_M \circ \eta(t)} f)(Z_M(s)) ds,
  %\label{eq:mdpt_martingale}
\end{equation*}
so long as the expectations exist. The scaled version of \eqref{eq:pathwise_mdpt}, which is an approximation to $X^N$ satisfying \eqref{eq:scaled}, is
\begin{equation}
   Z^N_M(t) = Z^N_M(0) + \sum_k Y_k \left( N^{\gamma} \int_0^t  N^{c_k} \lambda_k\circ \rho (Z_M^N \circ \eta(s))
    ds  \right)\zeta_k^N, 
    \label{eq:mdpt_scaled}
\end{equation}
where now
\begin{equation*}
	\rho(z) = z + \frac{1}{2}h N^{\gamma} \sum_k N^{c_k} \lambda_k(z) \zeta^N_k.
\end{equation*}
 While we should write $\rho^N$ in the above, we repress the ``$N$'' in this case for ease of notation.  For  $\TG_z^N$ defined via  \eqref{eq:scaled_operator}, and $Z^N_M$ satisfying \eqref{eq:mdpt_scaled}, we have
\begin{equation*}
  \E f(Z_M^N(t)) = \E f(Z_M^N ( \eta(t) ) ) + \E \int_{\eta(t)}^t
  (\TG_{ \rho ( Z_M^N \circ \eta(t))}^N f)(Z_M^N(s)) ds,
\end{equation*}
for all $t>0$, so long as the expectations exist.

\subsubsection{The weak trapezoidal method.}   
We will now describe a trapezoidal type algorithm to approximate the solutions of
\eqref{eq:main} and/or \eqref{eq:scaled}.  The method was originally introduced in the work of Anderson and Mattingly in the diffusive setting, where it is best understood by using a path-wise representation that incorporates space-time white noise processes, see \cite{AndMatt2011}.  It has independently been extended to the current jump setting in \cite{Li2011b} where it was studied in the classical scaling $(\gamma = 0, \alpha_i \equiv 1, c_k \equiv 1$, with the step size coupled to the system size similarly to the analysis in \cite{AndGangKurtz2011}). 
%Similarly, the method below is best understood in the setting of jump processes when a Poisson random measure representation is utilized as opposed to the more common random time change representation \eqref{eq:RTC_exact}.

  In the algorithm below, which simulates a path up to a time $T>0$,
  % for $j \in \{1,2\}$ and $k,n \ge 0$, where $k$ enumerates over the reactions, $y_{k,n,j}(x_{k,n,j})$ are independent Poisson random variables with parameter $x_{k,n,j}$.  
  it is notationally convenient to define $[x]^+ = x\vee 0 = \max\{x,0\}$.

\vspace{1ex}

\begin{algorithm}[Weak trapezoidal method] \label{alg:weaktrap}
Fix $h>0$.  Set $Z(0) = x_0$, $t_0 = 0$, and $n = 0$.  Fixing a \, $\theta \in (0,1)$, we define
 \begin{align}
   \alp_1 \eqdef \frac{1}{2}\frac{1}{\theta(1 -
     \theta)} %= \frac{1}{2}(\frac{1}{\theta} + \frac{1}{1 - \theta})
   \qquad\text{and}\quad \alp_2 \eqdef
   \frac{1}{2}\frac{(1-\theta)^2 + \theta^2}{\theta(1 - \theta)}\,.
   % = \frac{1}{2} (\frac{\theta}{1 - \theta} + \frac{1 -
   %   \theta}{\theta}).
   \label{eq:xi}
 \end{align}
We repeat the following steps until $t_{n+1} = T$, in which we first
 compute a $\theta$-midpoint $y^*$, and then the new value $Z_\T (t_{n+1})$:
 \begin{enumerate}[$(i)$]
      \item Set $t_{n+1} = t_{n} + h$.  If $t_{n+1} \ge T$, set $t_{n+1} = T$ and $h = T-t_n$.
  \item For $k \in \{1,\dots,R\}$, let $\Lambda_{k,1} = \text{Poisson}(\lambda_k'(Z_{\T}(t_n))\theta h)$ be independent of each other and all previous random variables.
  \item  Set $ \ym = Z_\T (t_{n}) +  \sum_{k} \Lambda_{k,1}\zeta_k$. 
    \item For $k \in \{1,\dots,R\}$, let $\Lambda_{k,2} = \text{Poisson}([\alp_1 \lambda_k'(y^*) - \alp_2 \lambda_k'(t_n)]^+ (1-\theta) h)$ be independent of each other and all previous random variables.
  \item Set $Z_{\T}(t_{n+1}) = \ym + \sum_k \Lambda_{k,2} \zeta_k$.
   \item Set $n \leftarrow n+1$.
 \end{enumerate}
 \end{algorithm}

\begin{remark}
  Notice that on the $(n+1)$st-step, $\ym$ is the Euler
  approximation to $X(nh + \theta h)$ starting from $Z_\T (t_n)$ at
  time $nh$.
\end{remark}
\begin{remark}
  Notice that for all $\theta \in (0,1)$ one has $\xi_1 > \xi_2$
  and $\xi_1-\xi_2=1$.   
\end{remark}

We define the operator $\TG_{z_1,z_2}$ by
\begin{equation*}
   (\TG_{z_1,z_2} f)(x) \eqdef \sum_k  [\alp_1 \lambda_k'(z_1) - \alp_2\lambda_k'(z_2)]^+ (f(x + \zeta_k) - f(x)).
\end{equation*}
Then, for $\eta(t) \le t \le \eta(t) + \theta h$, the process $Z_{\T}$ satisfies
\begin{equation*}
  \E f(Z_\T(t)) = \E f(Z_\T ( \eta(t))) + \E \int_{\eta(t)}^t
  (\TG_{Z_\T ( \eta(t))} f)(Z_\T(s)) ds,
 % \label{eq:WT1}
\end{equation*}
where we recall that $\TG_z$ is defined via
\eqref{eq:tau_gen}, 
and for $\eta(t) + \theta h \le t \le \eta(t) + h$, the process $Z_{\T}$ satisfies
\begin{equation*}
  \E f(Z_\T (t)) = \E f(Z_\T (\eta(t) + \theta h)) + \E \int_{\eta(t) + \theta h}^t
  (\TG_{Z_\T (\eta(t) + \theta h), Z_\T ( \eta(t))} f)(Z_\T(s)) ds.
\end{equation*}

 Finally, define the operator $\TG_{z_1,z_2}^N$ by
\begin{equation*}
   (\TG_{z_1,z_2}^N f)(x) \eqdef  (N^{\gamma} [\alp_1 \lambda(z_1) - \alp_2\lambda(z_2)]^+ \cdot \nabla^N ) f(x),
\end{equation*}
where for some $\theta \in (0,1)$, $\xi_1$ and $\xi_2$ satisfy \eqref{eq:xi}, and for $v \in \R^d$ the $i$th component of $v^+$ is $[v_i]^+ = \max\{v_i,0\}$.
Then, if $Z^N_\T$ represents the approximation to \eqref{eq:scaled} via the weak trapezoidal method,   for $\eta(t) \le t < \eta(t) + \theta h$
\begin{equation*}
  \E f(Z^N_\T(t)) = \E f(Z^N_\T ( \eta(t))) + \E \int_{\eta(t)}^t
  (\TG^N_{Z^N_\T ( \eta(t))} f)(Z^N_\T(s)) ds,
\end{equation*}
whereas for $\eta(t) + \theta h \le t < \eta(t) + h$
\begin{equation*}
  \E f(Z^N_\T(t)) = \E f(Z^N_\T (\eta(t) + \theta h)) + \E \int_{\eta(t) + \theta h}^t
  (\TG^N_{Z^N_\T(\eta(t) + \theta h), Z^N_\T ( \eta(t))} f)(Z^N_\T(s)) ds.
\end{equation*}

%%%%%%%%%%%%%%%%%%%%%%%%%%%%

\section{Global Error from Local Error}
\label{sec:analysis}
Throughout the section, we will denote the vector valued process whose $i$th component satisfies \eqref{eq:scaled} by  $X^N$, and denote 
an arbitrary approximate process via $Z^N$. 
Also, we define the following semigroup operators acting on $f \in C_0(\R^d , \RR)$:
\begin{align}
\cP_t f(x) &\eqdef \E_{x} f(X^N(t)) \label{eq:semi_exact} \\
P_t f(x) &\eqdef \E_{x} f(Z^N(t) ),\notag
\end{align}
where for ease of notation we choose not to incorporate the notation $N$ into either $\cP_t$ or $P_t$.  Returning to the notation introduced in Section \ref{sec:intro}, we note that
\begin{equation*}
	B_f(Z^N,x,t) = \E_x f(X^N(t)) - \E_x f(Z^N(t)) = (\cP_t - P_t)f(x).
\end{equation*}
We may therefore interpret the difference between the above two operators, for  $t \in [0,T]$, as the weak error, or bias, of the approximate 
process $Z^N$ on the interval $[0,T]$.  As $h>0$ is our time-step, we note that  $B_f(Z^N,x,h)  = (\cP_h - P_h)f(x)$ is the 
one step local error.

\begin{definition}
Let $n$ be an arbitrary non-negative integer, and $\cM$ be a $m$ dimensional vector of $C(\RR^d,\RR)$ valued operators on 
$C(\RR^d , \RR)$,  with its $\ell$th coordinate denoted by $ \cM_{\ell}$.  Then we define 
\begin{equation*}
   \| f \|_n^\cM =  \sup \left \{ \left \| \left (\prod_{i=1}^p \cM_{\ell_i} \right)   f  \right \|_\infty,  1 \leq  \ell_i  \leq  m, ~p\leq n  \right\}.
\end{equation*}
\label{def:op_norm}
\end{definition}

\noindent For example, if $j, k, \ell \in \{1, ..., R \} $ then 
$$| (\nabla^N_j \nabla^N_k \nabla^N_\ell f)(x) |  \leq \|  f\|^{\nabla^N}_3,$$
where we recall that  $\nabla^N$ is defined in \eqref{eq:operators}.
Note that, for any $\cM, $
\begin{equation}
\| f\|^\cM_0 = \| f\|_0 = \| f\|_{\infty}.
\label{eq:0norm}
\end{equation}
Also note that, by definition, for $n \ge 0$
$$ \| f \|_n^{\cM}  \leq  \| f \|_{n+1}^{\cM}.$$

\begin{definition}
Suppose $\cM : C(\RR^d, \RR) \to C(\RR^d, \RR^R)$  and 
 $Q : C(\RR^d, \RR) \to C(\RR^d, \RR)$ are operators. Then define
\begin{align*}
\| Q  \|_{j \to \ell}^\cM &\eqdef \sup_{f \in C^j,  f \neq 0} \frac{\| Q f \|_\ell^{\cM}}{\|f \|_j^{\cM}}. 
\end{align*} %\label{opnorm}
\end{definition}

Note that as stated in the Introduction, the purpose of this paper is to derive bounds for the global weak error of the different approximate processes, which, due to \eqref{eq:0norm}, consists of deriving bounds for
$\|(P_h^n  - \cP_{nh}) \|^{\cM}_{m \to 0}$, for an appropriately defined $\cM$ and a reasonable choice of $m\ge 0$.   Theorem \ref{thm:LG} below quantifies how the global error $\|(P_h^n  - \cP_{nh}) \|_{m \to 0}^\cM$ can  be 
bounded using the one-step local error $\|P_h - \cP_h  \|^{\cM}_{m\to 0}$.  In Section \ref{sec:local}, we will  derive the requisite bounds for the 
local error for each of the three methods. 

\begin{thm}
Let $\cM$ be a $C(\R^d,\R^R)$ valued operator on 
$C(\RR^d , \RR)$. 
%Suppose $Z$ is an approximation of the process $X$ with 
%$$P_h f (x) \eqdef E_{x}[f(Z(h))].$$ 
Then for any $n,m \ge 0$,  and $h>0$
$$\|(P_h^n  - \cP_{nh}) \|^\cM_{m \to 0}  = O (n\ \| P_h - \cP_h   \|^{\cM}_{m \to 0} \max_{\ell \in \{1, ... , n \} }\{\| \cP_{\ell h }\|_{m \to m}^\cM \})   $$
\label{thm:LG}
\end{thm}

\begin{proof} 
Let $f \in C_0(\R^d,\R)$.
Note that, since $\| g \|_0 = \| g  \|_0^{\cM} $ for any $g$,  
\begin{align*}
\| P_{h}^{j-1} \|_{0 \to 0}^{\cM}  \| P_{h} - \cP_h \|_{m \to 0}^{\cM}  &=\| P_{h}^{j-1} \|_{0 \to 0} \| P_{h} - \cP_h \|_{m \to 0}^{\cM} .  
\end{align*}
With this in mind
\begin{align*}
\begin{split}
\| (P_{h}^n - \cP_{nh} ) f \|_{0 }  &=\big \| \sum_{j=1}^{n}  (P_{h}^j \cP_{h(n-j)} - P^{j-1}_h \cP_{h(n-j+1)} ) f \big\|_0 \\
&\le \sum_{j=1}^{n} \| P_{h}^{j-1} (P_h - \cP_h) \cP_{h(n-j)}  f \|_{0}  \\
%&\leq \sum_{j=1}^{n-1} \frac{\| P_{h}^j (P - \cP) \cP_{h(n-j)}  f \|_0}{ \|(P - \cP) \cP_{h(n-j)} f \|_0}                \frac{  \|(P - \cP) \cP_{h(n-j)} f \|_0  }{  \| \cP_{h(n-j)} f \|_3}   \| \cP_{h(n-j)} f  \|_3\\
&\leq \sum_{j=1}^{n} \| P_{h}^{j-1} \|_{0 \to 0}  \| P_{h} - \cP_h \|_{m \to 0}^{\cM}  \| \cP_{h(n-j)}\|_{m \to m}^{\cM} \|f \|_m^{\cM} .
%\label{global} 
\end{split}
\end{align*}
Since $P_h$ is a contraction, i.e. $\| P_{h} \|_{0 \to 0}\leq 1$, the result is shown.
\end{proof}

The following result, where $\nabla^N$ replaces $\cM$ in Theorem \ref{thm:LG}, is now immediate.
\begin{cor}\label{thm:global_bound}
Under the same assumptions of Theorem \ref{thm:LG} and with $f \in C_0^m(\R^d,\R)$, 
\begin{equation*}
   \|(P_h^n  - \cP_{nh}) f \|_{0}^{\nabla^N} = O (n \|P_h - \cP_h  \|^{\nabla^N}_{m\to 0}  \max_{\ell \in \{1,\dots,n\}}  \{ \| \cP_{\ell h} f \|_m^{\nabla^N}  \} ). 
\end{equation*}
\end{cor} 

\noindent The following generalization, which allows for variable step sizes, is straightforward.  

\begin{cor}
For $f \in C_0^m(\R^d,\R)$
\begin{equation*}
   \|\E_x f(Z_{t_n})  -  \E_x f(X_{t_n}) \|_\infty = O (n \max_{i=1,...,n}\{ \|P_{h_i} - \cP_{h_i}  \|^{\nabla^N}_{m\to 0} \} \max_{\ell \in \{1,\dots,n\}}  \{ \| \cP_{t_\ell} f \|_m^{\nabla^N}  \} ). 
\end{equation*}
\end{cor}

Thus, once we compute the local one step error $\|P_h - \cP_h  \|^{\nabla^N}_{m\to 0}$ for an approximate process,  we have a bound on the global weak error that depends only on the 
semigroup $\cP_t$ of the
original process.  We will delay discussion of $\| \cP_{t} f \|_m^{\nabla^N} $ for now, as this term is independent of the approximate process.
  Instead, in the next section we provide bounds for $ \|P_h - \cP_h  \|^{\nabla^N}_{m\to 0}$ for each of the three methods described in Section \ref{sec:numerics}.

%\begin{thm}
%Choose $m,n \ge 0$.  Then for any $m > 0$, there exists $C_m$ such that 
%\begin{equation*}
%   \|(P_h^n  - \cP_{nh}) \|_{m \to 0}^{\nabla^N} = O (n \|P_h - \cP_h  \|^{\nabla^N}_{m\to 0}   e^{C_m N^{\gamma} t}) ,
%\end{equation*}
%where 

%  $$	C_n = 2\left( \|\lambda\|^{\nabla^N}_1 n \ R  + R(n- 1) \|\lambda\|_{n}^{\nabla^N}\right). $$
%  \label{thm:global_bound}
%\end{thm} 

\section{Local errors}
\label{sec:local}

Section \ref{sec:calculus} will present some necessary propositions and lemmas.  Sections \ref{sec:euler}, \ref{sec:mdpt}, and \ref{sec:trap} will present the local analyses of the Euler, midpoint, and weak trapezoidal methods, respectively. 
\subsection{Analytical tools}
\label{sec:calculus}

%\begin{definition}
%Let $n$ be arbitrary positive integer, and $\cM$ be a $m$ dimensional vector of $R$ valued operators with its $\ell$th coordinate denoted by $ \cM_{\ell}$, then define 
%$$\| f \|_n^\cM =  \max\left \{ \left \| \left (\prod_{i=1}^p \cM_{\ell_i} \right)   f  \right \|_\infty,  1 \leq  \ell_i  \leq  m, ~p\leq n  \right\}.$$
%\label{def:op_norm}
%\end{definition}

%
%\begin{definition}
%Denote the $j$th directional derivative of $f$ into the direction $[v_1, v_2, ... v_j ]$ by $f'[v_1, ..., v_j]$ and 
%\begin{align}
%\| f \|_j   \eqdef \sup_{x}   \{  f'[v_1, ...., v_j](x),  \|v\| = 1    \}  
%\label{eq:norm}
%\end{align}
%\end{definition}

\begin{prop}
Let $f  \in C^1_0(\RR^d , \RR^R)$.  For any $k \in \{1, \dots, R\}$  
\begin{equation*}
   \nabla^N_k  f \in O (N^{c_k - \m_k } \|f  \|_1 ) \subset  O (1). 
\end{equation*}
In particular, $N^{-c_k}\nabla^N_k  f  $ is bounded.   \label{derivbound}
\end{prop}

\begin{proof}
The result follows from the fact that for any $w \in \R^d$
\begin{equation*}
	|f(x + w) - f(x)| \le |w| \|f\|_1.
\end{equation*}
\end{proof}

Define,  for any multi-subset $I$ of $\{1, ... , R \},  $
$$\nabla^N_I  f \eqdef  \left \{ ( \prod_{i=1}^{|I|} \nabla^N_{\ell_i}  ) f  \right \} ,$$
so that,
$$\| f\|_n^{\nabla^N} = \sup_{|I| \leq n} \| \nabla^N_I  f \|_{\infty}.$$

\begin{prop}\label{prop:deriv}
Let $f  \in C^j_0(\RR^d , \RR^R)$. Then, 
\begin{equation*}
	\|f\|_j^{\nabla^N} = O(\|f\|_j).
\end{equation*}
%$$  \| f \|_j^{\nabla^N}  \in O(N^{(c_k - \min\{ \alpha_i \})j } \| f \|_j).$$   \\ 
\end{prop} 

\begin{proof}
   The case $j = 1$ follows from Proposition \ref{derivbound}.  Now consider $\nabla_I^Nf(x)$ for a multi-set $I$ of $\{1,\dots,R\}$, with $|I| =j \ge 2$. If $\m_k > 0$ for all $k \in I$, the statement is clear. If on the other hand, $\m_k =0$ for some $k\in I$, then for this specific $k$, we have $c_k \le 0$ and
\begin{equation*}
   \| \nabla_I^N f\|_{\infty} \le 2N^{c_k} \|\nabla_{I\backslash k}^N f\|_{\infty} = O(\|f\|_{j-1}) = O(\|f\|_j),
   \end{equation*}
   where the second to last equality follows by an inductive hypothesis.
\end{proof}

We make some definitions associated with $\nabla^N. $ Let $g:\R^d \to \RR^R$. For $i,j \in \{1,\dots,R\}$
\begin{align}
 \label{eq:notations}
\begin{split}
[D^N g(x)]_{ij} &\eqdef \nabla_j^N g_i(x) \\
[(\nabla^N)^2 ]_{ij} &\eqdef  \nabla_i^N \nabla_j^N  \\ 
diag(N^c)  &\eqdef diag(N^{c_1}, ...,N^{c_R}).
\end{split} 
\end{align}
Also, we define $\mathbf{1}_R$ to be the $R$ dimensional vector whose entries are all $1$.
\begin{lemma} (Product Rule)\label{lemma:product}
Let $g, q: \RR^d \to \RR^R$  be  vector valued functions. Then
\begin{align*}
\begin{split}
\nabla^N_k ( g  \cdot  q)(x) &=   (\nabla^N_k g \cdot  q)(x) +    (g \cdot   \nabla^N_k q)(x)  + 
 N^{-c_k} (\nabla^N_k  g \cdot  \nabla^N_k q)(x).
 \end{split}
 \end{align*}
 Also,   
\begin{align*}
\begin{split}
\nabla^N ( g  \cdot q )(x) &= [D^N g]^T q (x)  + [ D^N q]^T g (x)+  
 diag(N^c)^{-1} ( [D^N  g]^T  \times  [D^N q]^T  )(x)    \mathbf{1}_R f.
\end{split}  
\end{align*}
\end{lemma}

\begin{proof}
Note that, for any $k$, 
\begin{align*}
\begin{split}
\nabla^N_k (g \cdot q)(x)  =&   N^{c_k}(g(x+\zeta_k^N)q(x + \zeta_k^N) - g(x)q(x))\\
&= N^{c_k}(g(x + \zeta_k^N) - g(x)) q(x) +  N^{c_k}(q(x + \zeta_k^N) - q(x)) g(x)    \\
& ~~~~~~~+ N^{-c_k}  N^{c_k}(q(x + \zeta_k^N) - q(x)) N^{c_k} (g(x + \zeta_k^N) - g(x)) \\
&= (\nabla^N_k g) \cdot q )(x)  + (\nabla^N_k q \cdot g )(x) + N^{-c_k}  (\nabla^N_k g \cdot \nabla^N_k q )(x),
 \end{split} 
 \end{align*}
 verifying the first statement.  To verify the second, one simply notes that the above calculation holds for every coordinate, and the result follows after simple bookkeeping. 
\end{proof}

\begin{cor}
Let $\lambda: \RR^d \to \RR^R$  be a vector valued function,  and $f : \RR^d\to \RR.$ Then
\begin{align*}
\nabla^N_k ( \lambda  \cdot  \nabla^N f )(x) &=  ( \nabla^N_k \lambda  \cdot  \nabla^N) f +   \lambda  \cdot    \nabla^N \nabla^N_k f  + 
 N^{-c_k} \nabla^N_k  \lambda  \cdot   \nabla^N\nabla^N_k f .
 \end{align*}
 Also,   
\begin{align}
\begin{split}
\nabla^N (\lambda \cdot \nabla^N f ) &= [D^N \lambda]^T \nabla^N f + [(\nabla^N)^2  f] \lambda +  
 diag(N^c)^{-1} ( [D^N \lambda  \times (\nabla^N)^2]   \mathbf{1}_R f.
\end{split} 
 \label{product3}
\end{align}
\end{cor}
\begin{proof}
Simply put $g = \lambda $ and  $q = \nabla^N f$,  and note that $\nabla^2$ is symmetric.
\end{proof}

\subsection{Euler's method} 
\label{sec:euler}

Throughout subsection \ref{sec:euler}, we let $Z_E^N$ be the Euler approximation to $X^N$, and let
\begin{equation*}
	P_{E,h} f(x) \eqdef \E_x f(Z_E^N(h)),
\end{equation*}
where $h$ is the step-size taken in the algorithm. Below, we will assume $h < N^{-\gamma}$, which is a natural stability condition, and is discussed further in Section \ref{sec:stability}.
\begin{thm} \label{thm:local_Euler}
Suppose that the step size $h$ satisfies 
$h <  N^{- \gamma}$. Then
\begin{equation*}
\| P_{E,h}  - \cP_h\|_{2 \to 0}^{\nabla^N} =  O(N^{2\gamma}  h^2).
\end{equation*}
\end{thm}

\begin{proof}
For Euler's method with initial condition $x_0$, 
\begin{equation}
  P_{E,h} f(x_0) = f(x_0) + h \TG_{x_0}^N f(x_0) + \frac{h^2}{2} (\TG^N_{x_0})^2 f(x_0)  + O (N^{3 \gamma} \| f\|^{\nabla^N}_3 h^3),
  \label{eq:RE5}
\end{equation}
where, noting $\nabla^N  \lambda(x_0) = 0$ and using the product rule in Lemma \ref{lemma:product}, we have
\begin{align}
\TG_{x_0}^N f  &=  N^{\gamma} \lambda (x_0)  \cdot \nabla^N f \notag  \\
(\TG_{x_0}^N)^2 f &= N^{\gamma} \lambda (x_0)  \cdot \nabla^N  ( N^{\gamma} \lambda (x_0)  \cdot \nabla^N f) \notag\\
&= N^{2\gamma} \lambda (x_0)^T  [(\nabla^N)^2 f ]  \lambda (x_0). \label{eq:RE5.5}
\end{align}
On the other hand, for  the exact process \eqref{eq:scaled},  
\begin{align}
\cP_h f(x_0) &= f(x_0) + h \EG^N f(x_0) + \frac{h^2}{2} (\EG^N)^2 f(x_0)  + O (N^{3 \gamma} \| f\|^{\nabla^N}_3 h^3),
\label{eq:RE6}
\end{align}
where, again,  
\begin{equation*}
	\EG^N f = N^{\gamma} \lambda\cdot \nabla^N f.
\end{equation*}
Noting that, 
\begin{align}
\begin{split}
(\EG^N)^2 f (x) &= N^{2\gamma} (\lambda  \cdot \nabla^N ( \lambda \cdot \nabla^N f(x)) )  \\
&=  N^{2\gamma}  \lambda^T ( [D^N \lambda] ^T \nabla ^N f(x) + [(\nabla^N)^2 f] \lambda(x) + 
N^{2\gamma} \lambda^T ( diag(N^{-c}) [D^N \lambda \times (\nabla)^2 ] 1_{R}f )  
%&=  N^{2\gamma}  \lambda^T [D^N \lambda] ^T \nabla ^N f(x) + 
 %N^{2\gamma}   \lambda^T [(\nabla^N)^2 f] \lambda (x) + 
 %O( N^{2\gamma-\min\{\m_k\}}  \| f\|^{\nabla^N}_2),
 \label{Exact}
\end{split}
\end{align}

\noindent and defining 
\begin{align*}%\label{eq:abc}
\begin{split}
a(x) &\eqdef N^{2\gamma}  \lambda^T [D^N \lambda] ^T \nabla ^N f(x)\\
b(x) &\eqdef N^{2\gamma}   \lambda^T [(\nabla^N)^2 f] \lambda (x) \\ 
c(x) &\eqdef N^{2\gamma} \lambda^T [ diag(N^{-c}) [D^N \lambda \times (\nabla^N)^2 ] 1_{R}f(x) ],
\end{split}
\end{align*}

\noindent we can write 
\begin{align*}
\begin{split}
\cP_h f(x_0) &=f(x_0) + h \EG^N f(x_0) +   \frac{h^2}{2} (a(x_0) + b(x_0) + c(x_0))  + O (N^{3 \gamma}  \| f\|^{\nabla^N}_3  h^3).
\end{split}
\end{align*}
Note that  $\TG^N_{x_0} f(x_0) = \EG^N f(x_0)$ and $b(x_0) = (\TG_{x_0}^N)^2f(x_0)$. We may then compare \eqref{eq:RE5} and \eqref{eq:RE6}
\begin{align*}
\begin{split}
(P_{E,h} - \cP_h) f(x_0) &= \frac{h^2}{2} ((\TG_{x_0}^N)^2f(x_0)- (a(x_0) + b(x_0) + c(x_0)))  + O(N^{3 \gamma}\| f\|^{\nabla^N}_3 h^3 ) \\
&= \frac{h^2}{2}(-a(x_0)- c(x_0)) +  O(N^{3 \gamma}\| f\|^{\nabla^N}_3 h^3 ).
\end{split}
\end{align*}
The term $a(x) + c(x) = O(N^{2\gamma} \|f \|_2^{\nabla^N}  )$ is clearly non-zero in general, giving the desired result.
\end{proof}

\subsection{Approximate midpoint method}
\label{sec:mdpt}
Throughout subsection \ref{sec:mdpt}, we let $Z^N_M$ be the midpoint method approximation to $X^N$, and let
\begin{equation*}
	P_{M,h} f(x) \eqdef \E_x f(Z_M^N(h)),
\end{equation*}
where $h$ is the step-size taken in the algorithm.  As before, we will assume $h < N^{-\gamma}$, which is a natural stability condition, and is discussed further in Section \ref{sec:stability}.

\begin{thm}\label{thm:mdpt_local}
Suppose that the step size $h$ satisfies 
$h <  N^{- \gamma}$. Then
\begin{equation*}
\| (P_{M,h}  - \cP_{h}) \|_{3 \to 0}^{\nabla^N} =  O( N^{3\gamma} h^3 +  N^{2\gamma - \min \{ \m_k \}} h^2   ).
\end{equation*}
\end{thm}

\begin{remark}
  Theorem \ref{thm:mdpt_local} predicts that the midpoint method behaves locally like a third order method and globally like a second order method
 if $h$ is in a regime satisfying $N^{\gamma} h \gg N^{-\min\{\m_k\}}$, or equivalently if $h \gg N^{-\gamma - \min\{\m_k\}}$.  
This agrees with the result found in \cite{AndGangKurtz2011} pertaining to the midpoint method, which had $\gamma = 0$, $\m_k \equiv 1$, 
and the running assumption that $h \gg 1/N$. This behavior  is demonstrated via numerical example in Section \ref{sec:examples}. 
  \end{remark}
  
\begin{proof} (of Theorem \ref{thm:mdpt_local})
 Let $\zeta^N$ denote the matrix with $k$th column $\zeta_k^N$, i.e.
\begin{align*}
[\zeta^N] &= [\zeta_1^N, \zeta_2^N, ...., \zeta_R^N].
\end{align*}
Recall that $\rho$ is defined via 
\begin{equation*}
\rho(z) =z + \frac{h}{2} N^{\gamma} \sum_k \lambda_k(z) N^{c_k} \zeta^N_k.
\end{equation*}
After some algebra, we have  
\begin{align*}
\TG^N_{\rho(x_0)} f(x) &= N^{\gamma} (\lambda (x_0 + \frac{h}{2} N^{\gamma} \sum_k \lambda_k(x_0) N^{c_k} \zeta^N_k) ) \cdot \nabla^N f(x) \\
%&= N^{\gamma} (\lambda (x_0 + \frac{h}{2} N^{\gamma} [\zeta^N] diag(N^c) \lambda(x_0)     ) \cdot \nabla^N f(x) \\
%&= N^{\gamma}  (\lambda (x_0) +     \frac{h}{2} N^{\gamma} [D\lambda(x_0) ]  [\zeta^N] diag(N^c) \lambda(x_0) ) \cdot \nabla^N f(x)  + O(N^{2\gamma    } \|f \|_1^{\nabla^N} h^2)\\
%&=  N^{\gamma} \lambda (x_0)  \cdot \nabla^N f(x) + N^{2\gamma} \frac{h}{2}   [D\lambda(x_0) ]  [\zeta^N] diag(N^c) \lambda(x_0)  \cdot \nabla^N f(x) + O(N^{2\gamma} \|f \|_1^{\nabla^N}  h^2  ) \\
&= N^{\gamma} \lambda (x_0)  \cdot \nabla^N f(x) + w(x_0) + O(N^{2\gamma} \|f \|_1^{\nabla^N}  h^2  ).
\end{align*}
\noindent where
$$w(x) \eqdef N^{2\gamma} \frac{h}{2}   [D\lambda(x_0) ]  [\zeta^N] diag(N^c) \lambda(x_0)  \cdot \nabla^N f(x). $$

\noindent Next, using the product rule \eqref{product3}, we see
\begin{align*}
    (\TG^N_{\rho(x_0)})^2 f(x)  &= N^{\gamma} \lambda (x_0 + \frac{h}{2} [\zeta^N] diag(N^c) \lambda(x_0)  ) \cdot 
\nabla^N (N^{\gamma}  \lambda (x_0 + \frac{h}{2} [\zeta^N] diag(N^c) \lambda(x_0)  )  \cdot \nabla^N f)(x)  \\
&=N^{2\gamma} \lambda (x_0 + \frac{h}{2} [\zeta^N] diag(N^c) \lambda(x_0)  )^T [(\nabla^N)^2 f] 
\lambda (x_0 + \frac{h}{2} [\zeta^N] diag(N^c) \lambda(x_0)  )  \cdot \nabla^N f)(x)  \\
%&= N^{2 \gamma} \lambda(x_0)^T [(\nabla^N)^2  f (x)] \lambda(x_0) + O(N^{2\gamma} \|f \|^{\nabla^N}_2 h )\\
&= g(x_0) + O(N^{2\gamma} \|f \|^{\nabla^N}_2 h ),
\end{align*}
\noindent where
\begin{equation*}
  g(x_0) \eqdef  N^{2 \gamma} \lambda(x_0)^T [(\nabla^N)^2  f (x)] \lambda(x_0). 
\end{equation*}
Therefore, since $N^{\gamma} \lambda(x_0)\cdot \nabla^Nf(x_0) = \cA^N f(x_0)$, it follows that
\begin{align*}
\begin{split}
     P_{M,h} f(x_0) = f(x_0) &+ h \TG_{\rho(x_0)}^N f(x_0) + \frac{h^2}{2} (\TG^N_{\rho(x_0)})^2 f(x_0)  + O (N^{3 \gamma} \|f \|_3^{\nabla^N}h^3)\\
                       = f(x_0) &+ h\left ( \EG^N f(x_0)  +  w(x_0) +      O (N^{2 \gamma} \|f \|_2^{\nabla^N} h^2 )\right) \\
                       &+ \frac{h^2}{2} \left ( g(x_0) + O (N^{2 \gamma} \|f \|_2^{\nabla^N} h) \right  )   + O (N^{3 \gamma} \|f \|_3^{\nabla^N}h^3). 
\end{split} 
\end{align*}

\noindent Recall that
\begin{align*}
\begin{split}
(\EG^N)^2 f (x) &=  a(x) + b(x) + c(x),  
\end{split}
\end{align*}
where %$a,b,c$ are defined in \eqref{eq:abc},
\begin{align}
a(x) &= N^{2\gamma}  \lambda^T [D^N \lambda] ^T \nabla ^N f(x),\notag \\
b(x) &= N^{2\gamma}  \lambda^T [(\nabla^N)^2 f] \lambda (x), \notag \\
c(x) &= N^{2\gamma}  \lambda^T [ diag(N^{-c}) [D^N \lambda \times (\nabla^N)^2 ] 1_{R}f(x)],\label{eq:cx0}
\end{align}
and
\begin{align*}
\cP_h f(x_0) &= f(x_0) + h\EG^N f(x_0) +   \frac{h^2}{2} (a(x_0) + b(x_0) + c(x_0)) + O (N^{3 \gamma} \|f \|_3^{\nabla^N}h^3).
\end{align*}

\noindent Noting that $b(x_0) = g(x_0)$, we see
\begin{align}
 \label{eq:pre_result}  
\begin{split}
(P_{M,h} - \cP_h) f(x_0) &=  hw(x_0) + \frac{h^2}{2} \left(g(x_0) - (a(x_0)+b(x_0)+c(x_0) )\right) + O (N^{3 \gamma} \|f \|_3^{\nabla^N}h^3)  \\
&= (hw(x_0) - \frac{h^2}{2}a(x_0)) - \frac{h^2}{2} c(x_0) + O (N^{3 \gamma} \|f \|_3^{\nabla^N}h^3) .
\end{split}
\end{align}
We will now gain control over the terms $(hw(x_0) - \frac{h^2}{2}a(x_0))$ and $\frac{h^2}{2} c(x_0)$, separately. 

Handling $\frac{h^2}2 c(x_0)$ first,  we have that $\nabla^N \lambda_k \in O(N^{c_k - \m_k}),$ and so
\begin{equation*}
c(x_0) = O(N^{2\gamma - \min \{\m_k\}} \|f\|_2^{\nabla^N}  ).   
\end{equation*}

\noindent Next, we will show that
\begin{equation*}
     hw(x_0) - \frac{h^2}{2} a(x_0) =  O(N^{2 \gamma  -\min\{\m_k \} } \| f\|_1^{\nabla^N}  h^2 ).
\end{equation*}

\noindent We have 
\begin{align}
\begin{split}
hw(x_0) - \frac{h^2}{2} a(x_0) &= 
\frac{h^2}{2}N^{2\gamma}  [D\lambda(x_0) ]  [\zeta^N] diag(N^c) \lambda(x_0)  \cdot \nabla^N f(x_0) - \frac{h^2}{2} N^{2\gamma}  \lambda^T [D^N \lambda] ^T \nabla ^N f(x) \\
&= \frac{h^2}{2} N^{2\gamma} \bigg(  [D\lambda(x_0) ]  [\zeta^N] diag(N^c) - [D^N \lambda(x_0)] \bigg) \lambda(x_0) \cdot  \nabla^N f(x_0).
  \label{deriv} 
\end{split}
\end{align}
By Proposition \ref{prop:deriv}, $\nabla^N f(x)$ is bounded by $\|f \|^{\nabla^N}_1$.  Therefore, we just need to show that the difference between the two square matrices
\begin{align}
 [D^N \lambda(x_0) ]   \qquad \text{and} \qquad  [D\lambda(x_0) ]  [\zeta^N] diag(N^c)
 \label{eq:side}
 \end{align}
 is $O(N^{-\min\{\m_k \} })$. 
Recalling the definitions in \eqref{eq:notations}, the $(i, j)$th entry of the left side of \eqref{eq:side} is 
$$N^{c_j} ( \lambda_i (x_0 + \zeta_j^N )  - \lambda_i(x_0)) $$
whereas that of the right side of \eqref{eq:side} is  
$$N^{c_j}\nabla \lambda_i  \cdot \zeta_j^N.$$
Also, note that, for $ \lambda \in C^2_c(\RR^d ,  \RR)$, 
$$( (\lambda (x + v) - \lambda (x))  - \nabla \lambda(x)  \cdot v) \in O(|v|^2 \|\lambda \|_2 ).$$

\noindent where $$\|\lambda \|_2 = \sup\{\|\lambda\|_\infty, \|\partial_{x_i} \lambda\|_\infty, \|\partial_{x_j} \partial_{x_\ell} \lambda\|_\infty, i,j,k \leq d.\}$$ 
\noindent Since $\|\lambda_k \|_2$ is bounded for any $k$,  the difference between the $(i,j)$th entries of the two expressions in \eqref{eq:side} is 
\begin{equation*}
  O(N^{c_j} N^{ -  2\m_j}). 
\end{equation*}
\noindent Also, recall that $c_j - \m_j  \leq 0$. Thus the above is also
 $$O(N^{ - \min\{ \m_k\}}).$$
\noindent Therefore \eqref{deriv} is of order 
 $$O(N^{2 \gamma  -\min\{\m_k \} } h^2 \|f \|^{\nabla^N}_1),$$
\noindent as desired.
  Combining the above with \eqref{eq:pre_result} gives us
 \begin{align}
 \begin{split}
 \|(\cP_h  - P_{M,h})f\|_0 &= O(N^{2 \gamma  - \min \{ \m_k\} } \| f\|^{\nabla^N}_1 h^2 + N^{2 \gamma - \min \{ \m_k \} } \|f \|_2^{\nabla^N}  h^2 + N^{3\gamma} \| f\|^{\nabla^N}_3 h^3  ) \\
%& = O ( \| f\|^{\nabla^N}_3 [ N^{3\gamma}h^3 + N^{2 \gamma} h^2( N^{ - \min \{ \m_k \} } + N^{ - \min \{ \m_k \} } )]  ) \\
&= O ( \| f\|^{\nabla^N}_3 [ N^{3\gamma}h^3 + N^{2 \gamma - \min \{ \m_k \}} h^2 ]  ), 
\end{split}
 \end{align} 
 implying 
 $$\| P_{M,h} - \cP_h  \|_{3\to 0}^{\nabla^N} = O(N^{3\gamma}h^3 + N^{2 \gamma - \min \{ \m_k \}} h^2),   $$
 as desired.
 \end{proof}

\subsection{Weak trapezoidal method}
\label{sec:trap}

Throughout subsection \ref{sec:trap}, we let $Z_\T^N$ be the weak trapezoidal approximation to $X^N$, and  let
\begin{equation*}
	P_{\T,h} f(x) \eqdef \E_x f(Z_\T^N(h)),
\end{equation*}
where $h$ is the size of the time discretization.  We will again only consider the case $h<N^{-\gamma}$, which is a natural stability condition and is discussed further in Section \ref{sec:stability}.

We make the standing assumption that for all  $x$ in our state space of interest, and $  k, j \in \{1,\dots, R\},$ we have 
\begin{align}
  	\xi_1 \lambda_k(x + \zeta_j^N) - \xi_2\lambda_k(x) &\ge 0,
	\label{eq:stand_assump}
\end{align}
where $\xi_1>\xi_2$ are defined in \eqref{eq:xi} for some $\theta \in (0,1)$.  Noting that $\zeta_j^N$ will often be small, and that $\xi_1 - \xi_2=1$, for most processes, including those arising from biochemistry, the requirement \eqref{eq:stand_assump} holds so long as the process is not directly at the boundary of the positive orthant.  
Weakening (5.12) is almost certainly doable, for example by gaining
control over the probability that a process leaves a region in which
the condition holds.  This is an avenue for  future work.

\begin{thm} \label{thm:local_trap}
 Suppose that the step size $h$ satisfies 
$h <  N^{- \gamma}$. Then 
$$  \| (P_{\T,h}  - \cP_{h}) \|_{3 \to 0}^{\nabla^N} = O( N^{3\gamma} h^3).$$
\end{thm}

\begin{proof}
Consider one step of the method with a step-size of size $h$ and with initial value $x_0$.  Note that the first step of the algorithm produces 
a value $y^*$ that is distributionally equivalent to one produced by a  Markov process with generator $B_1^N$ given by 
\begin{align*}
B_1^N f (x) =   N^{\gamma}  \lambda(x_0)  \cdot {\nabla^N} f(x).   
\end{align*}
Next, given both $x_0$ and $y^*$, step 2 produces a value which is distributionally equivalent to one produced by a Markov process with generator
\begin{align}
\begin{split}
B_2^N f(x) =   N^{\gamma} [\alp_1 \lambda(y^*) - \alp_2 \lambda(x_0) ]^+   \cdot {\nabla^N} f(x).   
\end{split}
\end{align}
 Recall that for the exact process, 
\begin{align*}
\cP_h f(x_0)&= f(x_0) + h\cA^N f(x_0) + \frac{ h^2 }{2} (\cA^N)^2 f(x_0) + O(N^{3\gamma}  \|f\|^{\nabla^N}_3 h^3).
\end{align*}
For the approximate process we have,
\begin{align}
P_{\T,h} f(x_0) &= \E_{x_0} [\E_{x_0}[f(Z^N_\T(h)) | y^*] ] \notag\\
% = \E_{x_0} f(y^*) + \E_{x_0} \left [ \E_{y^*} \left[ \int_0^{(1-\theta)h} B_2 f(Z^N_{\T}(s)) ds  \right] \right ]\notag \\
%&= \E_{x_0} f(y^*) + (1 -\theta) h \E_{x_0}[ B_2 f(y^*) ]+ \frac{(1 -\theta)^2 h^2}{2} \E_{x_0} [B^2_2 f(y^*)]  + O(\|\cB_2^3 f\|_\infty h^3)\\
&= \E_{x_0} f(y^*) + (1 - \theta)h  \E_{x_0}[ B_2^N f(y^*) ]+ \frac{(1 -\theta)^2 h^2}{2} \E_{x_0} [(B^N_2)^2 f(y^*)]  + O(N^{3\gamma}\|f\|^{\nabla^N}_3 h^3).\label{eq:to_expand}
\end{align}
We will expand each piece of \eqref{eq:to_expand} in turn.  Noting that $B_1^N f(x_0) = \cA^N f(x_0)$, the first term is
\begin{align*}
\E_{x_0} f(y^*) &= f(x_0) + \E_{x_0} \left[ \int_0^{\theta h} B_1^N f (Z_s) ds \right] \\%  = f(x_0) + \theta h B_1f(x_0) + \frac{\theta^2h^2}{2}  B_1^2 f(x_0) + O(\|\cB_1^3 f\|_\infty h^3)\\
%&=f(x_0) +  \theta h \cA f(x_0) + \frac{\theta^2h^2}{2}  B_1^2 f(x_0) + O(\|\cB_1^3 f\|_\infty h^3) \\
&=f(x_0) +  \theta h \cA^N f(x_0) + \frac{\theta^2h^2}{2}  (B_1^N)^2 f(x_0) + O(N^{3\gamma}\| f \|_3^{\nabla^N} h^3). 
\end{align*}

\noindent We turn attention to the second term, $(1 - \theta)h  \E_{x_0}[ B_2^N f(y^*) ]$, and begin by making the following definition:
\begin{align*}
g(y^*) &\eqdef B_2^N  f(y^*) =    N^{\gamma} [\xi_1 \lambda(y^*) - \xi_2 \lambda(x_0) ]^+   \cdot {\nabla^N} f(y^*), 
\end{align*}
so that $g(x) =    N^{\gamma} ( [\xi_1 \lambda(x) -  \xi_2 \lambda(x_0) ]^+  \cdot {\nabla^N} )f(x)$.  Because $\xi_1 - \xi_2 = 1$, we have 
\begin{equation*}
   g(x_0) =   N^{\gamma} \lambda(x_0)   \cdot {\nabla^N} f(x_0) = \cA^N f(x_0).
\end{equation*}
By our standing assumption \eqref{eq:stand_assump}
\begin{align*}
 g(x_0 + \zeta_k) - g(x_0)&=  
 % ( [\xi_1 \lambda -  \xi_2 \lambda(x_0) ]^+  \cdot {\nabla^N} )f(x_0+\zeta_k) -   N^{\gamma} ( [\xi_1 \lambda -  \xi_2 \lambda(x_0) ]^+  \cdot {\nabla^N} )f(x_0)\\
      N^{\gamma} (\xi_1 \lambda(x_0 + \zeta_k) - \xi_2\lambda(x_0))\cdot {\nabla^N} f(x_0 + \zeta_k) -   N^{\gamma} \lambda(x_0) \cdot {\nabla^N} f(x_0).
\end{align*}
After some algebra
\begin{align*}
  B_1^N g(x_0) &=   N^{\gamma}  (\lambda(x_0)\cdot {\nabla^N} g)(x_0)  =   N^{\gamma}\sum_k N^{c_k}\lambda_k(x_0)[g(x_0 + \zeta_k) - g(x_0)]\\
%   &=   N^{\gamma}\sum_k N^{c_k} \lambda_k(x_0)   \bigg[ N^{\gamma} (\xi_1 \lambda(x_0 + \zeta_k) - \xi_2\lambda(x_0))\cdot {\nabla^N} f(x_0 + \zeta_k) - N^{\gamma} \lambda(x_0) \cdot {\nabla^N} f(x_0)\bigg]\\
 %  &= \xi_1   N^{\gamma} \sum_k N^{c_k} \lambda_k(x_0)   \bigg[N^{\gamma}\lambda(x_0 + \zeta_k) \cdot {\nabla^N} f(x_0 + \zeta_k) - N^{\gamma}\lambda(x_0) \cdot {\nabla^N} f(x_0) \bigg]\\
 %  &\hspace{.2in} - \xi_2 N^{\gamma} \sum_k N^{c_k} \lambda_k(x_0)  \bigg[   N^{\gamma}\lambda(x_0)\cdot {\nabla^N} f(x_0 + \zeta_k) -  N^{\gamma}\lambda(x_0)\cdot {\nabla^N} f(x_0) \bigg]\\
  &= \xi_1 N^\gamma \lambda(x_0) \cdot  \nabla^N  (N^\gamma  \lambda \cdot  f)(x_0)   - \xi_2 N^\gamma \lambda(x_0) \cdot \nabla^N (  \lambda(x_0) \cdot f)(x_0) \\
   &= \xi_1 (B_1^N \cA^N f(x_0)) - \xi_2 ((B_1^N)^2 f)(x_0).
\end{align*}

\noindent Thus, 
%considering $\nabla^N f$ and $N^\gamma$ in the expression of $g$, 
\begin{align*}
\E_{x_0}[B_2^N  f(y^*)] &= \E_{x_0} [g(y^*)] = g(x_0) +  \theta h B_1^N g (x_0) + O(N^{3\gamma} \| f\|_3^{\nabla^N} h^2  ) \\
%&= \cA f (x_0) + \theta h  (\lambda(x_0)  \cdot {\nabla^N} g)(x_0) + O(\| b\|^2_\infty   \| f \|^{\nabla^N} _2 h^2  )  \\
%&=  \cA f (x_0) + \theta h  \left[ \lambda(x_0)  \cdot {\nabla^N} (\alp_1 b \cdot {\nabla^N} f)(x_0) -   \lambda(x_0)  \cdot {\nabla^N} (  \alp_2 \lambda(x_0) \cdot {\nabla^N} f)(x_0) \right]  + O(\| b\|^2_\infty  \| f \|^{\nabla^N} _2 h^2  )\\
%&=  \cA f (x_0) + \theta h  \left[ \lambda(x_0)  \cdot {\nabla^N} (\alp_1 \cA f)(x_0) -   \lambda(x_0)  \cdot {\nabla^N} ( \alp_2 \lambda(x_0) \cdot {\nabla^N} f)(x_0) \right]  + O(\| b\|^2_\infty  \| f \|^{\nabla^N} _2 h^2  )\\
&= \cA^N f(x_0) + \theta h  \left[ \alp_1 (B_1^N \cA^N f)(x_0) -  \alp_2 (B_1^N)^2 f(x_0) \right] + O(N^{3\gamma} \| f \|^{\nabla^N} _2 h^2  )\\
&= \cA^N f(x_0) + \theta h  \left[ \alp_1  (\cA^N)^2 f (x_0) - \alp_2 (B_1^N)^2 f(x_0) \right] + O(N^{3\gamma} \| f \|^{\nabla^N} _3 h^2  ),\end{align*}
where the last line follows since $B_1^N f (x_0) = \cA^N f(x_0)$ for any $f$.

 Finally, we turn the the last term in \eqref{eq:to_expand}.  Define 
\begin{align*}
q(y^*) &\eqdef (B_2^N)^2  f(y^*) \\
&=  [\alp_1 \lambda(y^*) -  \alp_2 \lambda(x_0) ]^+   \cdot {\nabla^N}  ( [\alp_1 \lambda(y^*) - \alp_2 \lambda(x_0) ]^+ {\nabla^N} f ) (y^*), 
\end{align*}
so that 
\begin{equation*}
   q(x) =[\alp_1 \lambda - \alp_2 \lambda(x_0) ]^+   \cdot {\nabla^N}  ( [\alp_1 \lambda - \alp_2 \lambda(x_0) ]^+ {\nabla^N} f )(x).
 \end{equation*}
By our standing assumption \eqref{eq:stand_assump} we have
\begin{align}
\begin{split}
\E_{x_0} [(B_2^N)^2 f(y^*)] &= \E_{x_0} [q(y^*)] \\
&= q(x_0) + O(N^{3\gamma} \|  f \|_3^{\nabla^N} h  ) \\
&= (B_1^N)^2 f  (x_0) + O(N^{3\gamma} \|  f \|_3^{\nabla^N} h  ). 
\end{split}
\end{align}
Noting that
\begin{equation*}
(1-\theta)\theta \xi_1 = \frac{1}{2} \qquad \text{and} \qquad  (1-\theta)\theta \xi_2 = \frac{(1 - \theta)^2 + \theta^2}{2},
\end{equation*}
we may conclude the following from the above calculations 
\begin{align*}
\E_{x_0}[f( Z_{\T,h}^N) ] &=  \E_{x_0}f(y^*) + (1-\theta) h \E_{x_0} [B_2^N f(y^*) ] + \frac{(1 - \theta)^2 h^2}{2} \E_{x_0} [(B_2^N)^2 f(y^*)] \\
&\hspace{.3in}+ O(N^{3\gamma}  \|f \|^{\nabla^N}_3h^3) \\
%&=f(x_0) +  \theta h  \cA f(x_0) + \frac{\theta^2 h^2}{2} B_1^2 f(x_0) + O(N^{3\gamma} \|f \|^{\nabla^N}_3h^3) \\
%&+(1-\theta) h \cA f(x_0) +  (1-\theta) \theta h^2 \alp_1  \cA^2 f(x_0) - (1 -\theta) \theta h^2 \alp_2 B_1^2 f(x_0) + O(N^{3\gamma} \|f \|^{\nabla^N}_3 h^3)\\
%&+\frac{(1- \theta)^2 h^2}{2} B_1^2 f(x_0) + O(N^{3\gamma} \|f \|^{\nabla^N}_3h^3) \\
& = f(x_0) + \theta h  \cA^N f(x_0) + \frac{\theta^2 h^2}{2} (B_1^N)^2 f(x_0)\\
&\hspace{.2in} +(1-\theta) h \cA^N f(x_0) +  \frac{h^2}{2}  (\cA^N)^2 f(x_0) - \frac{h^2}{2} [(1 - \theta)^2 + \theta^2] (B_1^N)^2 f(x_0) \\
&\hspace{.2in} +\frac{(1- \theta)^2 h^2}{2} (B_1^N)^2 f(x_0) + O(N^{3\gamma} \|f \|^{\nabla^N}_3h^3) \\
&= f(x_0) +  \cA^N f(x_0) +  \frac{h^2}{2} (\cA^N)^2 f(x_0) + O(N^{3\gamma} \|f \|^{\nabla^N}_3h^3).
\end{align*}
Thus 
\begin{equation*}
   \| (P_{\T,h} - \cP_h )  f \|_0 \in O(N^{3\gamma} \| f \|_3^{\nabla^N} h^3),
   \end{equation*}
and the proof is complete.
\end{proof}

\section{Global Bounds and Stability}
\label{sec:regularity}

In Section \ref{sec:global} we bound $\|\cP_t f\|^{\nabla^N}_n $, which was the remaining piece to handle in Theorem \ref{thm:LG} to give us global bounds on the weak error induced by the different methods.  In Section \ref{sec:stability}, we briefly discuss some issues related to stability of the different methods.

\subsection{Bounds on $\|\cP_t f\|^{\nabla^N}_n $}
\label{sec:global} 
In this section we bound  $\|\cP_t f\|^{\nabla^N}_n $, where $n$ is a nonnegative integer and $\cP_t$ is the semigroup operator \eqref{eq:semi_exact} of the scaled process \eqref{eq:scaled}.  We point out, however, that for any process
 $X^N$ for which $\cP_t$ is well behaved, in that $\|\cP_t\|_{n\to 0}^{\nabla^N}$ is bounded uniformly in $N$, the following results are not needed, and, in fact, would most likely be a \textit{least optimal} bound, as the bound grows exponentially in $N^{\gamma}t$. 
% For many, though certainly not all, models of interest, we will find that $\gamma \le 0$. 
%The results below point out that the main theorems of this paper are most applicable when $\gamma \leq 0$. 
Note that any system satisfying the classical scaling has $\gamma = 0$.  We also point out that the arguments used below are quite similar to those used in \cite{Li2011} by Hu, Li, and Min, which were extensions of those used in \cite{AndGangKurtz2011} by Anderson, Ganguly, and Kurtz.

For $t\ge 0$ and any $x\in \R^d$, We define 
\begin{equation*}
  v(t,x) \eqdef \cP_t f(x) = \E_x f(X^N_t).
\end{equation*}

\begin{thm}\label{thm:semi_bound}
If $ \|f\|_n^{\nabla^N} < \infty$, then 
\begin{equation*}
  \| v(t, \cdot) \|_n^{\nabla^N}  = \|\cP_t f\|^{\nabla^N}_n  \le \|f\|_n^{\nabla^N} e^{ N^{\gamma} C_n  t}
  \end{equation*}
  where
  \begin{equation}
  	C_n = 2\left( \|\lambda\|^{\nabla^N}_1 n \ R  + R(n- 1) \|\lambda\|_{n}^{\nabla^N}\right).
	\label{eq:Cn}
  \end{equation}
 \end{thm}

We delay the proof of Theorem \ref{thm:semi_bound} until the following Lemma is shown, the proof of which is similar to that found in \cite{Li2011}, which itself was an extension of the proof of Lemma 4.3 in \cite{AndGangKurtz2011}.

\begin{lemma}
 Given a multiset $I$ of $\{1, \cdots , R \} , $ there exists a function $q_I(x)$ that is a linear function of terms of the form $\nabla_{J}^N v(t,x)$ with $|J| < |I|$, so that
 \begin{equation*}
 	\del_t \nabla_I^N v(t,x) = N^{\gamma} (\lambda\cdot \nabla^N)\nabla_I^N v(t,x) +  N^{\gamma} \sum_{i = 1}^{|I|} (\beta_i \cdot \nabla^N) \nabla_{I\backslash \ell_i} v(t,x + \zeta_{\ell_i}) + N^{\gamma}q_I(x),
 \end{equation*}
 where $\beta_i = \nabla_{\ell_i}^N \lambda$.  Further, $q_I$ consists of at most $R(|I| - 1)$ terms of the form $\nabla_{J}^N v(t,x)$, each of whose coefficients are bounded above by $\|\lambda\|_{|I|}^{\nabla^N}$.
 \label{lem:deriv}
 \end{lemma}
 
 \begin{proof}
This goes by induction. For $|I| = 0$, the statement follows because 
\begin{equation}
   \del_t v (t,x) =  N^{\gamma} (\lambda \cdot \nabla^N) v(t,x).
   \label{eq:zero_case}
\end{equation}
Note that in this case, there are no $\beta_i$ or $q$ terms.
It is instructive to perform the $|I| = 1$ case.  We have
\begin{align*}
\del_t \nabla^N_k v (t,x) &= \nabla^N_k \del_t v (t,x) \\
&=  \nabla^N_k ( N^{\gamma}  \lambda \cdot \nabla^N v(t,x)) \\
&=     N^{\gamma} (  \nabla^N_k \lambda \cdot  \nabla^N ) v(t,x) +  N^{\gamma}   \lambda  \cdot   \nabla^N_k \nabla^N v(t,x)  + 
 N^{\gamma} (N^{-c_k} \nabla^N_k  \lambda \cdot  \nabla^N_k \nabla^N v(t,x)).
 \end{align*}
 Note that for any $g:\R^d \to \R$ 
 \begin{align}
 	( \nabla^N_k \lambda  \cdot  \nabla^N) g(x) + (N^{-c_k} \nabla^N_k  \lambda  \cdot \nabla^N) \nabla_k^N g(x)  &= (\nabla_k^N\lambda  \cdot \nabla^N) g(x+\zeta_k).
	\label{eq:RE1}
 \end{align}
 Therefore, with $g(x) = v(t,x)$ in the above, we have
 \begin{equation*}
 	\del_t \nabla^N_k v (t,x) = N^{\gamma}  ( \lambda(x)  \cdot \nabla^N ) \nabla^N_k v(t,x) + N^{\gamma} (\nabla_k^N\lambda(x) \cdot \nabla^N) v(t,x+\zeta_k).
 \end{equation*}

Now assume that it holds for a set of size  $\leq |I|$.  Then, using the inductive hypothesis, Lemma \ref{product3}, and equation \eqref{eq:RE1} yields
\begin{align*}
	\del_t \nabla_k^N& \nabla_I^N v(t,x) \\
	&= \nabla_k^N \del_t \nabla_I^N v(t,x)\\
	&=N^{\gamma} \nabla_k^N \bigg[ (\lambda\cdot \nabla^N)\nabla_I^N v(t,x) +   \sum_{i = 1}^{|I|} (\beta_i \cdot \nabla^N) \nabla_{I\backslash \ell_i} v(t,x + \zeta_{\ell_i}) + q_I(x)\bigg]\\
%	&= N^{\gamma} \bigg[ \nabla^N_k \lambda  \cdot  \nabla^N \nabla_I^N v(t,x) +  \lambda  \cdot   \nabla^N  \nabla^N_k \nabla_I^N v(t,x)  + 
% N^{-c_k} \nabla^N_k  \lambda  \cdot   \nabla^N\nabla^N_k \nabla_I^N v(t,x) \bigg]\\
% & \hspace{.2in} + N^{\gamma} \sum_{i = 1}^{|I|} \bigg[  \nabla^N_k \beta_i  \cdot  \nabla^N \nabla_{I\backslash \ell_i}^N v(t,x+\zeta_{\ell_i}) +   \beta_i  \cdot    \nabla^N\nabla^N_k \nabla_{I\backslash \ell_i}^N v(t,x+\zeta_{\ell_i})\\
% &\hspace{.9in}  + 
% N^{-c_k} \nabla^N_k  \beta_i  \cdot   \nabla^N \nabla^N_k \nabla_{I\backslash \ell_i}^N v(t,x+\zeta_{\ell_i})\bigg]\\
% &\hspace{.2in} + N^{\gamma} \nabla_k^N q_I(x)\\
 &= N^{\gamma} \bigg[(\lambda  \cdot   \nabla^N)  \nabla^N_{I\cup k} v(t,x) +   ( \nabla^N_k  \lambda  \cdot   \nabla^N ) \nabla_I^N v(t,x + \zeta_k) \bigg]\\
 & \hspace{.2in} + N^{\gamma} \sum_{i = 1}^{|I|} \bigg[ (\beta_i  \cdot    \nabla^N) \nabla^N_k \nabla_{I\backslash \ell_i}^N v(t,x+\zeta_{\ell_i}) + (\nabla^N_k \beta_i  \cdot  \nabla^N) \nabla_{I\backslash \ell_i}^N v(t,x+\zeta_{\ell_i}+ \zeta_k)\bigg] \\
 &\hspace{.2in} + N^{\gamma} \nabla_k^N q_I(x)\\
 &= N^{\gamma}(\lambda  \cdot   \nabla^N)  \nabla^N_{I\cup k} v(t,x)  + N^{\gamma}\bigg[ ( \nabla^N_k  \lambda  \cdot   \nabla^N ) \nabla_{I\cup k \backslash k}^N v(t,x + \zeta_k) +  \sum_{i = 1}^{|I|}(\beta_i  \cdot    \nabla^N)  \nabla_{I\cup k \backslash \ell_i}^N v(t,x+\zeta_{\ell_i})\bigg]\\
 &\hspace{.2in} + N^{\gamma}\bigg[ \nabla_k^N q_I(x) + (\nabla^N_k \beta_i  \cdot  \nabla^N) \nabla_{I\backslash \ell_i}^N v(t,x+\zeta_{\ell_i} + \zeta_k)\bigg],
\end{align*}
showing the result.
\end{proof}

\begin{proof}(\textit{of Theorem \ref{thm:semi_bound}}\ )

Let $n \ge 0$.  Define  
\begin{equation*}
  U_n(t) \eqdef   \max_{x , |I| \le n  } | \nabla^N_{I }v( t , x ) | = \| v \|_n^{\nabla^N} . 
  \end{equation*}
Each $\nabla_I^N v(t,x)$ is a continuously differentiable function with respect to $t$.  Therefore, the maximum above is achieved at some $(I^*, x^*)$ for all $t \in [0,t_1]$ where $t_1 >0$.  Fixing this choice of $(I^*, x^*)$, we have 
\begin{equation*}
   U_n(t) = \nabla^N_{I^* }v( t , x^*)
\end{equation*}
for all $t < t_1$.
  
  Note that
\begin{align}
\begin{split}
  [(\lambda\cdot \nabla^N)\nabla_{I^*}^N& v(t,x^*)]\nabla_{I^*}^N v(t,x^*)  = \sum_k  \lambda_k(x) (\nabla_k^N \nabla_{I^*}^N v(t,x^*)) \nabla_{I^*}^N v(t,x^*)\\
  &= \sum_k N^{c_k} \lambda_k(x) (\nabla_{I^*}^N v(t,x^* + \zeta_k) - \nabla_{I^*}^N v(t,x^*)) \nabla_{I^*}^N v(t,x^*)\\
  &\le 0,
  \end{split}
  \label{eq:RE2}
\end{align}
where the final inequality holds by the specific choice of $I^*$ and $x^*$.
Also note that for any $\ell_i \in I^*$ and any choice of $x$
\begin{align}
| \nabla^N \nabla^N_{{I^*} \backslash \ell_i}  v(t, x)  |  \le  \sum_{k=1}^R |\nabla_k \nabla^N_{{I^*} \backslash \ell_i}  v(t, x)|  \leq R |\nabla^N_{I^*} v(t, x^*) |.
\label{eq:RE3}
\end{align}

\noindent From Lemma \ref{lem:deriv} and equations \eqref{eq:RE2} and \eqref{eq:RE3}, we have 
\begin{align}
\frac{1}{2} \del_t & (\nabla^N_{I^*} v(t, x^*))^2 =  (\del_t  \nabla^N_{I^*} v(t, x^*) ) \nabla^N_{I^*} v(t, x^*) \notag \\
&= N^{\gamma}\bigg[ (\lambda\cdot \nabla^N)\nabla_{I^*}^N v(t,x^*) +  \sum_{i = 1}^{|I^*|} (\beta_i \cdot \nabla^N) \nabla_{I^*\backslash \ell_i} v(t,x^* + \zeta_{\ell_i}) + q_{I^*}(x^*) \bigg] \nabla^N_{I^*} v(t, x^*)\label{eq:RE4}\\
&\le  N^{\gamma}\bigg[ \|\lambda\|^{\nabla^N}_1 |I^*| \ R \ |\nabla^N_{I^*} v(t, x^*)|^2 + R(|I^*| - 1) \|\lambda\|_{|I^*|}^{\nabla^N} |\nabla^N_{I^*} v(t, x^*)|^2\bigg],\notag
\end{align}
where we have used the fact that each $\beta_i = \nabla_{\ell_i} \lambda $ for $\ell_i \in I^*$.
Setting 
\begin{equation}
C_n =  2\left( \|\lambda\|^{\nabla^N}_1 n \ R  + R(n- 1) \|\lambda\|_{n}^{\nabla^N}\right),
\end{equation} 
we see by an application of Gronwall's inequality that the conclusion of the theorem holds for all $t < t_1$.  That is, for $t< t_1$
\begin{equation*}
 U_n(t)\le  \|f\|_n^{\nabla^N} e^{ N^{\gamma} C_n  t}.
\end{equation*}
To continue, repeat the above argument on the interval $[t_1, t_2)$, with $I^*,x^*$ again chosen to maximize $U_n$ on that interval, and note that 
\begin{equation*}
	U_n(t_1) \le \|f\|_n^{\nabla^N} e^{ N^{\gamma} C_n  t_1},
\end{equation*}
so that we may conclude that for $t_1\le t < t_2$,
\begin{equation*}
	U_n(t) \le \|f\|_n^{\nabla^N} e^{ N^{\gamma} C_n  t_1} e^{N^{\gamma} C_n  (t - t_1) }=  \|f\|_n^{\nabla^N} e^{ N^{\gamma} C_n  t}. 
\end{equation*}
Continuing on, we see that $t_i \to \infty$ as $i \to \infty$ by the boundedness of the time derivatives of $v(t,x)$, thereby concluding the proof.
\end{proof}

\begin{remark}
In the theorem above, $C_n \in \| \lambda \|_n^{\nabla^N}. $
\end{remark}

Combining all of the previous results, we have the following theorems.

\begin{thm} \textbf{(Global bound for the Euler method)} \\
Suppose that the step size $h$ satisfies 
$h <  N^{- \gamma},$  and $T = nh$. Then 
\begin{equation*}
\| (P_{E,h}^n  - \cP_{nh}) \|_{2 \to 0}^{\nabla^N} =  O(  N^{2\gamma}  h e^{C_2N^{\gamma} T} )
\end{equation*}
where $C_2 \in O(\| \lambda \|_2^{\nabla^N} )$ is defined in \eqref{eq:Cn}. 
\label{thm:euler_global}
\end{thm}

\begin{thm} \textbf{(Global bound for the midpoint method)}\\
Suppose that the step size $h$ satisfies 
$h <  N^{- \gamma},$  and $T = nh$. Then  
\begin{equation*} 
\| (P_{M,h}^n  - \cP_{nh}) \|_{3 \to 0}^{\nabla^N} =  O( [N^{3\gamma} h^2 +  N^{2\gamma - \min \{ \m_k \}} h ] e^{C_3N^{\gamma} T} )
\end{equation*}
where $C_3 \in O(\| \lambda \|_3^{\nabla^N} )$ is defined in \eqref{eq:Cn}. 
\label{thm:mdpt_global}
\end{thm}

The following immediate corollary to the theorem above recovers the result in \cite{AndGangKurtz2011}.
\begin{cor}
Under the additional condition  $h >  N^{-\gamma- \min\{\m_k\} }$ in Theorem \eqref{thm:mdpt_global},   the leading order of the error of the midpoint method is $O(h^2)$.
\end{cor}

\begin{thm} \textbf{(Global bound for the weak trapezoidal method)}\\
Suppose that the step size $h$ satisfies 
$h <  N^{- \gamma},$  and $T = nh$. Then  
\begin{equation*}
\| P^n_{\T,h} - \cP_{nh} \|_{3 \to 0 }^{\nabla^N} = O ( h^2 N^{3\gamma}e^{N^\gamma C_3  T} ) 
\end{equation*}
where $C_3 \in O(\| \lambda \|_3^{\nabla^N} )$ is defined in \eqref{eq:Cn}. 
\label{thm:wtm_global}
\end{thm}

%Note that, by mean value theorem, there exists a constant C such that
%\begin{align}
%\begin{split}
%\nabla^N_k f(x) &=  N^{\beta_k + \zeta_k \cdot \alpha-\gamma}(f(x + \zeta_k^{N})   - f(x)) \\
%&\leq N^{\beta_k + \zeta_k \cdot \alpha -\gamma}  C \sum_{i=1}^S N^{- \alpha_i} \| f \|_1 \\
%&\in O(\| f\|_1)
%\end{split} 
%\end{align}
%This holds for $N$ because 

Thus, we see that  the weak trapezoidal method detailed in Algorithm \ref{alg:weaktrap} is the only method that boasts a global error of second order in the stepsize $h$ in an
``honest sense.'' That is, it is a second order method regardless of the relation of $h$ with respect to $N$.  This is in contrast to the midpoint method which has  second order accuracy only when  the order of $h$ is larger than  $N^{-\gamma -\min\{\m_k\}}$.

\subsection{Stability Concerns}
\label{sec:stability}

The main results and proofs of our paper have incorporated stability concerns into the analysis.  This is seen in the statements of the theorems by the running condition that $h< N^{-\gamma}$, where we recall that $N^{\gamma}$ should be interpreted as the time-scale of the system.  Without this condition, the methods are unstable.  It is an interesting question, and the subject of future work, to determine the stability properties of other methods in this setting.  

As an instructive example, again consider the system
\begin{equation*}
	S_1 \overset{100}{\underset{100}{\rightleftarrows}} S_2
\end{equation*}
with $X_1(0) = X_{2}(0) = \text{10,000}$.  In this case, it is natural to take $N = $ 10,000.  As the rate constants are $100 = \sqrt{\text{10,000}}$, we take $\beta_1 = \beta_2 = 1/2$ and find that $\gamma = 1/2$.  The equation governing the normalized process $X^N_1$ is
\begin{equation*}
    X_1^N(t) = X_1^N(0) - Y_1\bigg(N^{1/2}N \int_0^t X_1^N(s)ds\bigg)\frac{1}{N} + Y_2\bigg(N^{1/2}N \int_0^t (2 - X_1^N(s))ds\bigg)\frac{1}{N}
\end{equation*}
where we have used that $X^N_1 + X^N_2 \equiv 2$.  It is now clear that if the condition $h < N^{-\gamma}$ is violated a path generated by any of the explicit methods discussed in this paper will behave quite poorly.

\section{Examples}
\label{sec:examples}

We provide two test systems.  The first is a simple linear system with three species that we will use to demonstrate our main analytical results.  The second is a gene-protien-mRNA model we will use to demonstrate the capabilities of the different methods on an actual test problem.  We note that in all simulations of the weak trapezoidal algorithm, we chose $\theta = 1/2$.
\vspace{.1in}

\noindent  \textbf{Example 1.}
Consider the following first order reaction network 

%%\begin{displaymath}

%%%%%%%%%%%%%%%%%%%%%%%%
%    0.03>    0.1>     
%  A < - > B < - > C   
%    < 1       1 >
%%%%%%%%%%%%%%%%%%%%%%%%

%\xymatrix{
%  A \ar@<lex>[r]^3 & \ar@<lex>[l]^5_{.} B}

%%\end{displaymath}

\begin{equation*}
A \overset{\kappa_1}{\underset{\kappa_2}{\rightleftarrows}} B \overset{\kappa_3}{\underset{\kappa_4}{\rightleftarrows}} C,
\end{equation*}
with $\kappa_1 = 0.03, \kappa_2 = 1, \kappa_3 = 0.1,$ and $\kappa_4 = 1$.
Starting from the initial state 
\begin{equation*}
  X(0) = (X_A(0), X_B(0), X_C(0)) = (13000, 100, 20),
  \end{equation*} 
  where we make the obvious associations $X_1 = X_A, X_2 = X_B,$ and $X_3 = X_C$.  We approximate $X(2)$ using the three methods considered in this paper: 
Euler, midpoint, and weak trapezoidal with a choice of $\theta = 1/2$. For first order systems, we 
may find the first moments and the covariances of $X(t)$ as solutions of linear ODEs using a Moment 
Generating function approach \cite{Othmer2005}.  
%For each method of approximation, we compare the first and the second moments of $Z(2)$ from those of $X(2)$, the latters of which can be numerically solved from the aforementioned ODEs. 

\begin{figure}
  \centering
  {\includegraphics[scale=0.3]{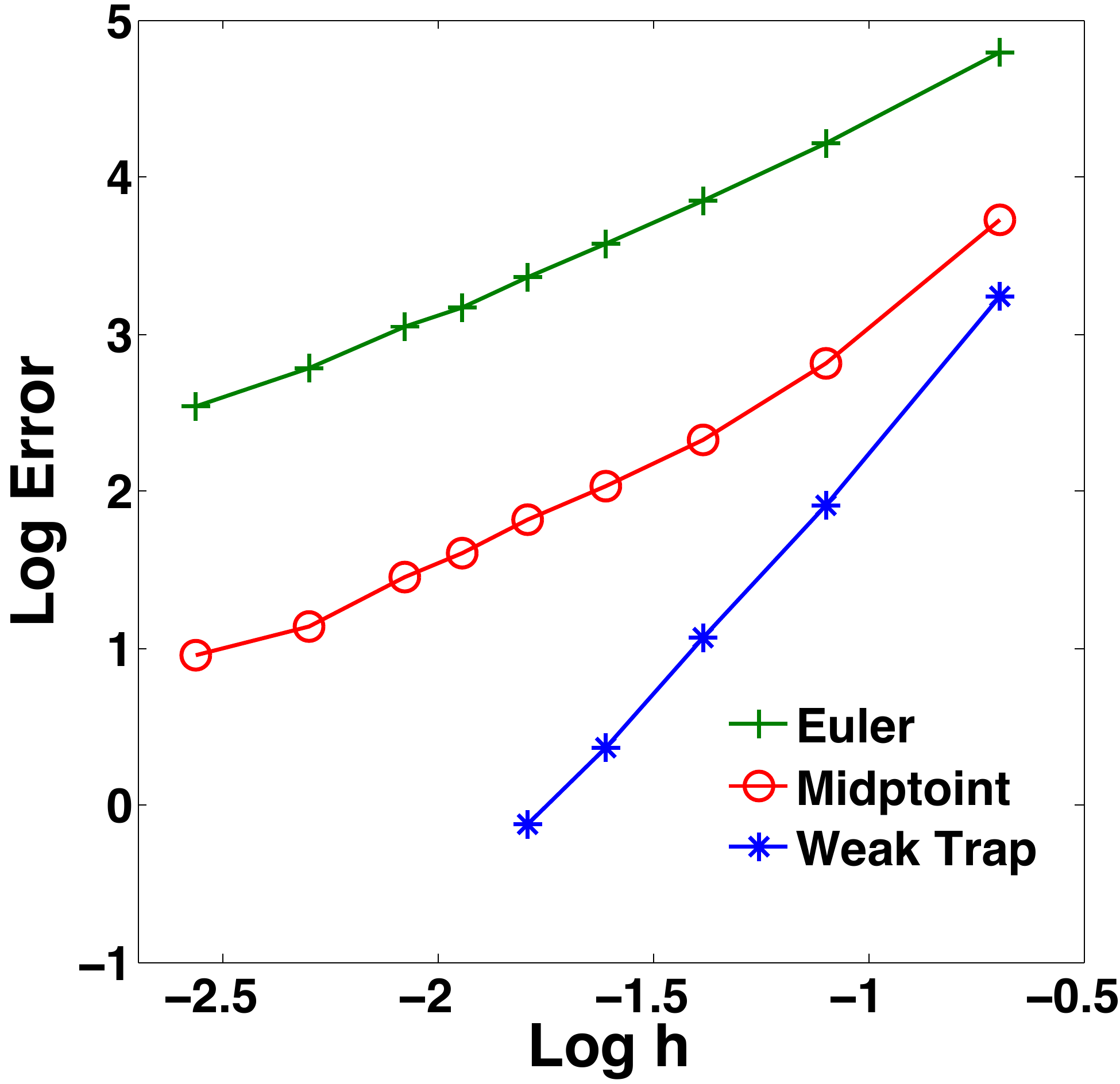}}
\caption{The log-log plot of  $|\E[X_3^2(2)] -  \E[Z_{3}^2(2)]|$ against $h$ 
for the three approximation methods.  The slope for Euler's method is 1.21, whereas the slope for the weak trapezoidal solution with $\theta = 1/2$ is 3.06, which is better than expected.  The curve governing the solution from the midpoint method appears to not be linear; a behavior predicted by Theorem \ref{thm:mdpt_local}.
} 
\label{fig:plotC2}
\end{figure}

In Figure \ref{fig:plotC2}, we show a log-log plot of  $|\E[X_3^2(2)] -  \E[Z_{3}^2(2)]|$ against $h$ for the three approximation methods.  Each data point was found from either $10^6, 2.9\times 10^6, 3.9\times 10^6, 4.9\times 10^6, 8\times 10^6,$ or $10^7$ independent simulations, with the number of simulations  depending upon the size of $h$ and the method being used.    The slope for Euler's method is 1.21, whereas the slope for the weak trapezoidal solution is 3.06, which is better than expected.  The curve governing the solution from the midpoint method appears to not be linear; a behavior predicted by Theorem \ref{thm:mdpt_local}.

\begin{figure}
  \centering
  \subfloat[]
  {\includegraphics[scale=0.3]{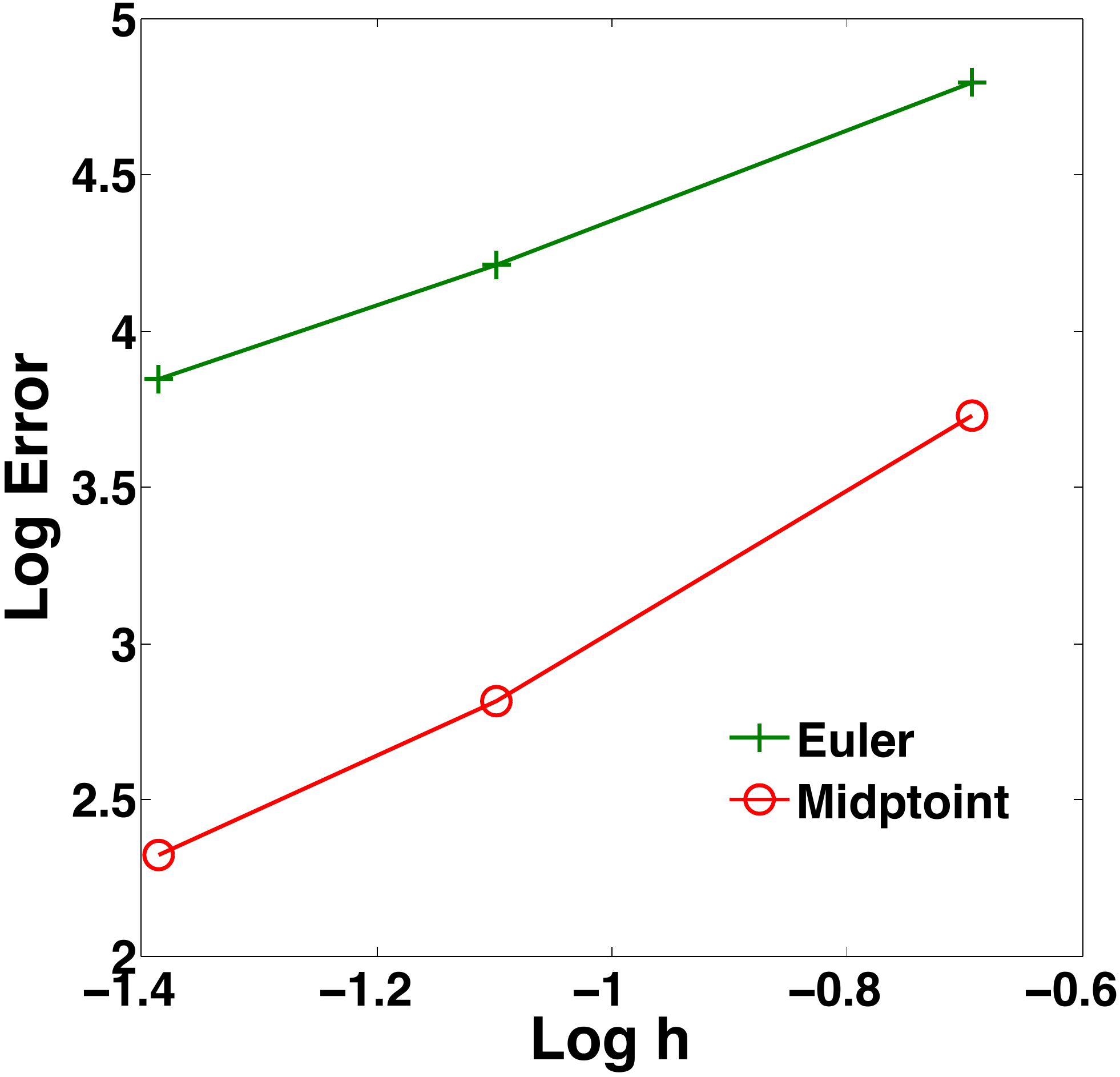}}
  \subfloat[]
  {\includegraphics[scale=0.3]{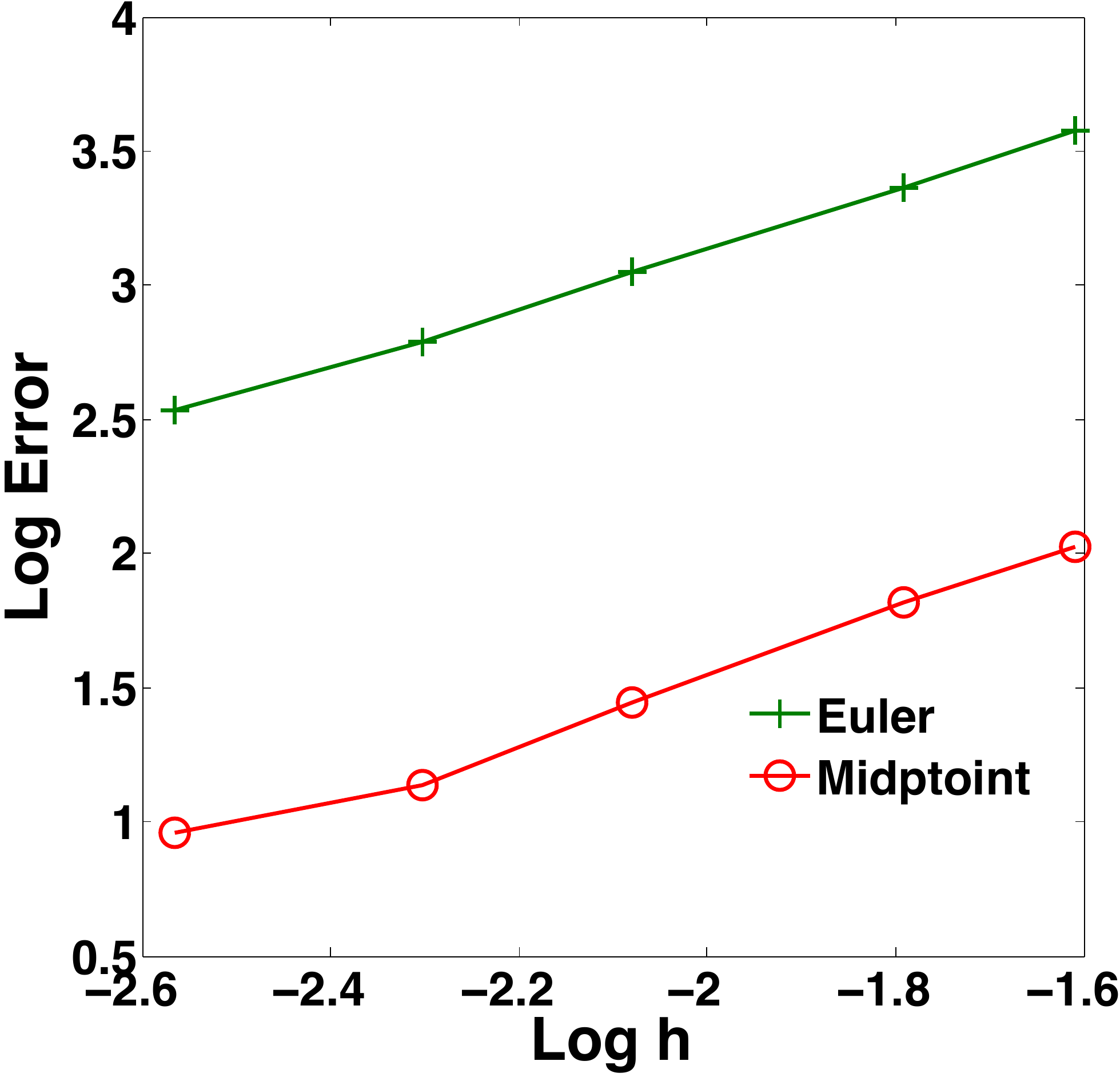}}  
 \caption{The log-log plot of $|\E[X_3^2(2)] -  \E[Z_3^2(2)]|$ against $h$.  The slope for generated via midpoint tauleaping shifts from $2.03$ in (a) to  $1.12$  in (b). }
\label{fig:plot2}
\end{figure}

In Figure \ref{fig:plot2} we again consider the log-log plots of  $|\E[X_3^2(2)] -  \E[Z_{3}^2(2)]|$ against $h$, but now only for Euler's method and the midpoint method so that we may see the change in behavior in the midpoint method predicted in Theorem \ref{thm:mdpt_local}.    In (a), we see that for larger $h$ the slope generated via the midpoint method is 2.03, whereas in (b) the slope is 1.12 when $h$ is smaller.  For reference, in (a) the slope generated by Euler's method is 1.366, whereas in (b) it is 1.09.

While the simulations make no use of the scalings inherent in the system, it is instructive for us to quantify them in this example so that we are able to understand the behavior of the midpoint method.    We have $N \approx 10^4$, $\alpha_1 = 1, \alpha_2 = 1/2,$  $\alpha_3 = 1/4,$ and $m_k = 1/4$.  Also, $\gamma \approx 0.$  Therefore, Theorem \ref{thm:mdpt_local} predicts  the midpoint method will behave as an order two method if $h \gg N^{-1/4}\approx 1/10$, or if $\log(h) \gg -2.3$, which roughly agrees with what is shown in Figures \ref{fig:plotC2} and \ref{fig:plot2}.  Note that  Theorem \ref{thm:mdpt_local} will never provide a sharp estimate as to when the behavior will change as it is a local result and the scalings in the system will change during the course of a simulation.

The fact that the trapezoidal method gave an order three convergence rate above does not hold in general.  This was demonstrated in the proof of Theorem \ref{thm:local_trap}, but it is helpful to also show this via example.
In Figure \ref{fig:Berrors} we present a log-log plot of $| \E X_2(2) - \E Z_2(2)|$ for the different algorithms on this same example.  The approximate slopes are: 1.02 for Euler's methods, 2.372 for midpoint method, and 2.3 for the trapezoidal method.  
%Note that Corollary \ref{cor:linear} explains why the midpoint method is as accurate as the weak trapezoidal algorithm in this case.  
We point out that all of the plots above represent results pertaining to the {\em non-normalized processes} as the simulation methods themselves make no use of the scalings.

\begin{figure}
\begin{center}
  {\includegraphics[scale=0.3]{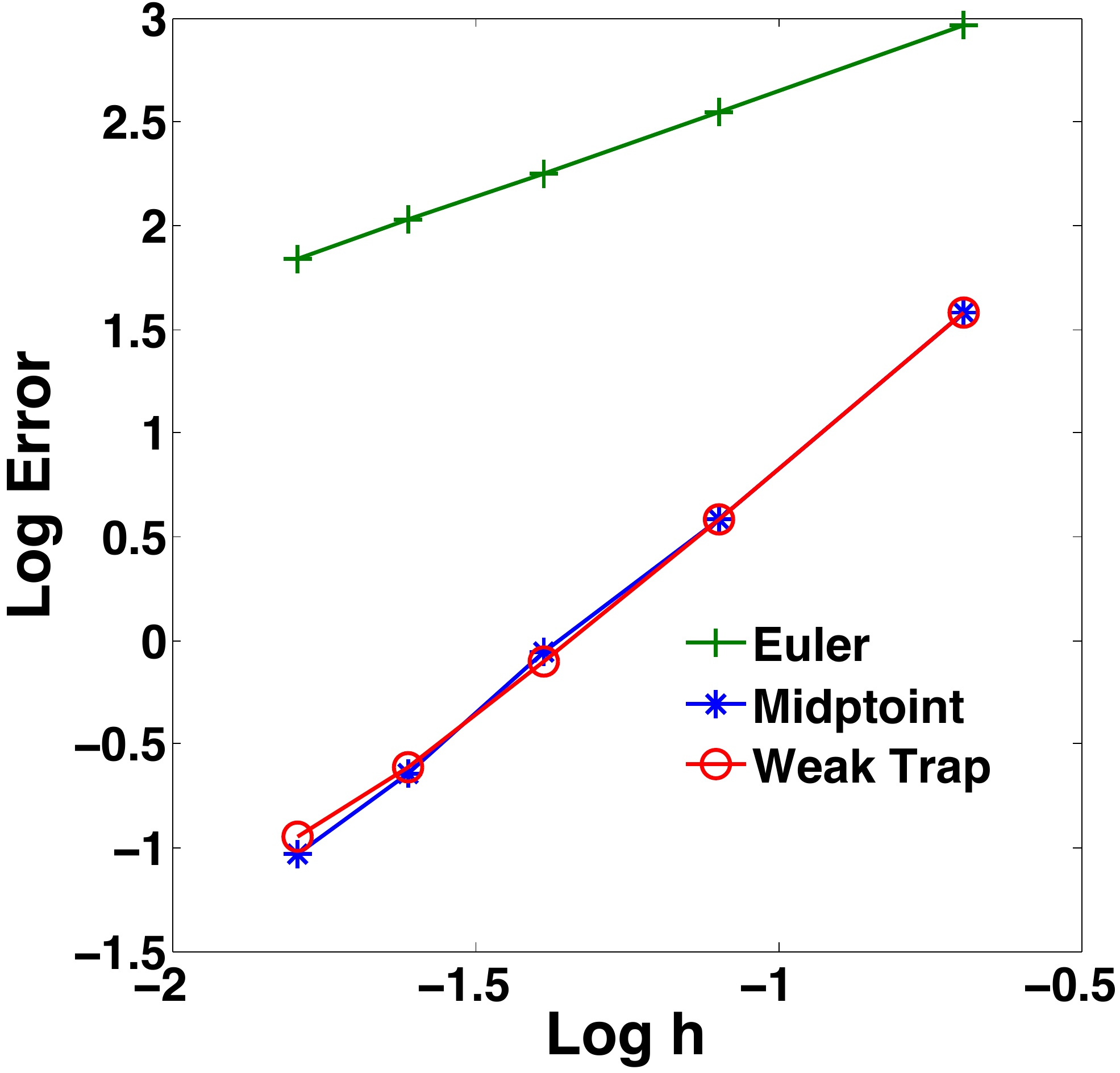}}
\end{center}
\caption{Log-log plot of  $|\E[X_3(2)] - \E[Z_3(2)]|$ against $h$ for the three approximation methods.  The approximate slopes are: 1.02 for Euler's methods, 2.372 for midpoint method, and 2.3 for the trapezoidal method.}
\label{fig:Berrors}
\end{figure}

\noindent \textbf{Example 2.}  Consider a  model of gene transcription and translation:
 \begin{align*}
\displaystyle G \overset{25}{\to} G + M, \quad M \overset{1000}{\to} M + P, \quad 2P \overset{0.001}{\to} D, \quad M \overset{0.1}{\to} \emptyset, \quad P \overset{1}{\to} \emptyset.  \end{align*}
Here a single gene is being translated into mRNA, which is then being transcribed into proteins, and finally the proteins produce stable dimers.  The final two reactions represent degradation of mRNA and proteins, respectively.  Suppose we start with one gene and no other molecules, and want to estimate the expected number of dimers at time $T=1$ to an accuracy of $\pm$ 1.0 with 95\% confidence.  Therefore, we want the variance of our estimator to be smaller than $(1/1.96)^2 = .2603$.
 
 While $\epsilon = 1$ for the unscaled version of this problem, the simulation of just a few paths of the system will show that there will be somewhere in the magnitude of 3,500 dimers at time $T=1$.  Therefore, for the scaled system, we are asking for an accuracy of $\widetilde \epsilon = 1/3500 \approx 0.0002857$.  Also, a few paths (100 is sufficient) shows that the order of magnitude of the variance of the normalized number of dimers is approximately 0.11.   Thus, the approximate number of exact sample paths we will need to generate can be found by solving
 \begin{equation*}
 	\frac{1}{n} \mathsf{Var}(\text{normalized \# dimers}) = (\widetilde \epsilon/1.96)^2 \implies n = 5.18 \times 10^6.
 \end{equation*}
 Therefore, we will need approximately five million independent sample paths generated via an exact algorithm.  Implementing the modified next reaction method \cite{Anderson2007a} on our machine (using Matlab), each path takes approximately 0.03 CPU seconds to generate.  Therefore, the approximate amount of time to solve this particular problem will be 155,000 CPU S, which is about forty three hours.  The outcome of such a simulation is detailed in Table \ref{table1}  where ``\# updates'' refers to the total number, over all paths, of updates to the system performed, and random variables generated, and is used as a quantification for the computational complexity of the different methods under consideration.  In terms of software and hardware, the authors used Matlab for all computations, which were performed on an Apple machine with a 2.2 GHz Intel i7 processor.
 
	\begin{table}
	\begin{tabular}{|c|c|c|c|c|}\hline
		 Approximation & \# paths &  CPU Time & Variance of estimator  & \# updates    \\  \hline
		 3714.2 $\pm$ 1.0 & 4,740,000  & 149,000  CPU S  & 0.25995  & 8.27 $\times 10^{10}$ \\ \hline
	\end{tabular}
	\caption{Performance of Exact algorithm with crude Monte Carlo estimator \eqref{eq:estimator1}.}
	\label{table1}
	\end{table}

Next, we solved the problem using Euler's method, the approximate midpoint method, and the weak trapezoidal method with $\theta = 1/2$.  We note that for each of the three approximations, we used the most naive implementation possible by simply setting the value of any component that goes negative in the course of a step to zero, and by using a fixed step size, $h>0$.  Thus, improvements can be gained on the stated results by using a more sophisticated implementation \cite{Anderson2007b,Cao2005}.  However, we did produce our approximate paths in batches of 50,000, which greatly reduces the cost of generating the Poisson random variables with the built in Matlab Poisson random number generator.

 In Table \ref{table2}  we provide data on the performance of Euler's method with various step-sizes, combined with a crude Monte Carlo estimator \eqref{eq:CMC}.  Note that the bias in Euler's method is apparent even for very small $h$.
%\noindent Method: Euler's method with crude Monte Carlo. 
\begin{table}
\begin{tabular}{|c|c|c|c|c|c|}\hline
Step-size & Approximation &  \# paths &  CPU Time & Variance of estimator  & \# updates    \\     \hline
 $h = 3^{-7}$&  3,712.3 $\pm$ 1.0 & 4,750,000 & 13,374.6 CPU S  & 0.25898  & $6.2 \times 10^{10}$  \\ [.1 ex] \hline
 $h = 3^{-6}$&  3,707.5 $\pm$ 1.0 &  4,750,000 & 6,207.9 CPU S & 0.25839  & $2.1 \times 10^{10}$ \\[.1 ex] \hline
 $h = 3^{-5}$&  3,693.4 $\pm$ 1.0 & 4,700,000 & 2,803.9 CPU S & 0.26018  & $6.9 \times 10^9$ \\ [.1 ex]\hline
 $h = 3^{-4}$& 3,654.6 $\pm$ 1.0 & 4,650,000 &  1,219 CPU S & 0.25940 & $2.6 \times 10^9$ \\ \hline
\end{tabular}
\caption{Performance of Euler's method with crude Monte Carlo.}
\label{table2}
\end{table}
In Table \ref{table3} we provide data on the performance of the midpoint method with various step-sizes, combined with a crude Monte Carlo estimator \eqref{eq:CMC}.  Note that the solution has a much higher variance when $h = 1/3$, thereby necessitating significantly more paths to get a desired tolerance.  This demonstrates the stability concerns discussed in Section \ref{sec:stability}.  This problem does not arise as much when using the weak trapezoidal method.
\begin{table}
\begin{tabular}{|c|c|c|c|c|c|}\hline
Step-size & Approximation &  \# paths &  CPU Time & Variance of estimator  & \# updates    \\     \hline
 $h = 3^{-4}$&  3,713.6 $\pm$ 1.0 & 4,650,000 &  1,269.1 CPU S  & 0.25996  &   $2.3 \times 10^9$ \\ [.1 ex] \hline
 $h = 3^{-3}$&  3,713.9 $\pm$ 1.0  & 4,500,000   & 497.5 CPU S & 0.25860  &  $7.6 \times 10^8$ \\[.1 ex] \hline
 $h = 3^{-2}$&  3,722.4 $\pm$ 1.0 & 4,050,000  & 177.6 CPU S & 0.25972  & $2.2 \times 10^8$  \\ [.1 ex]\hline
 $h = 3^{-1}$& 3,986.1 $\pm$ 1.0 & 18,500,000 & 376.0 CPU S & 0.26020  &  $3.3 \times 10^8$  \\ \hline
\end{tabular}
\caption{Performance of midpoint method with crude Monte Carlo.}
\label{table3}
\end{table}
In Table \ref{table4} we provide data on the performance of the weak trapezoidal method with various step-sizes, combined with a crude Monte Carlo estimator \eqref{eq:CMC}. 
%Note that while two random variables are required for each reaction channel per step, the CPU time for a given step size does not double when compared to the midpoint method, pointing out that it is not the generation of the random variables taking up the bulk of the CPU time.
 We see that for this example the midpoint method and the weak trapezoidal method are, overall, comparable. However, the weak trapezoidal method performs, in terms of bias and required CPU time, significantly better than does the midpoint method for $h = 1/3$. 
\begin{table}
\begin{tabular}{|c|c|c|c|c|c|}\hline
Step-size & Approximation &  \# paths &  CPU Time & Variance of estimator  & \# updates    \\     \hline
 $h = 3^{-4}$&  3,714.4 $\pm$ 1.0 & 4,750,000  & 2,120.5 CPU S  & 0.25940  &  $4.6 \times 10^9$ \\ [.1 ex] \hline
 $h = 3^{-3}$&  3,714.6 $\pm$ 1.0 & 4,750,000  & 898.2 CPU S & 0.25940  & $1.6 \times 10^9$  \\[.1 ex] \hline
 $h = 3^{-2}$& 3,725.6 $\pm$ 1.0  & 4,800,000 & 349.8 CPU S & 0.25965 & $5.2 \times 10^8$  \\ [.1 ex]\hline
 $h = 3^{-1}$& 3,673.3 $\pm$ 1.0 & 8,850,000 & 238.2 CPU S & 0.25944 & $3.2\times 10^8$  \\ \hline
\end{tabular}
\caption{Performance of weak trapezoidal method with $\theta = 1/2$, with crude Monte Carlo.}
\label{table4}
\end{table}

It is worth noting that both the midpoint and weak trapezoidal methods compare decently on this example with the multi-level Monte Carlo method developed recently for stochastic chemical kinetic systems \cite{AndersonHigham2011}.  The choice of which method (an explicit solver discussed herein or a multi-level Monte Carlo solver) a user wishes to implement will therefore often  be problem, and user, specific.

We next used each of the methods above to estimate the probability that the number of dimers at time 1 is greater than or equal to 6,000.  Note that this probability is the expected value of the indicator function $1_{\{X_{\text{Dimer}}(1) \ge 6,000\}}.$  The results are presented in Table \ref{table5}, which provides 95\% confidence intervals for a few choices of $h$ for each method.  Note that in computing this approximation the weak trapezoidal method has significantly less bias than does the midpoint method for comparable step-sizes, making it the method of choice for this  particular choice of function $f$.  The necessary CPU time for each of the methods is the same as those reported above.\vspace{.2in}

\begin{table}
\begin{center}
\begin{tabular}{|c|c|c|c|}\hline
Method & Step-size & \# paths & Approximation     \\     \hline
 Exact & N.A. & 4,520,000 & 0.02843 $\pm$ 0.00015  \\ [.1 ex] \hline
 Euler & $h = 3^{-7}$&  4,750,000 & 0.02818 $\pm$  0.00015  \\[.1 ex] \hline
 Euler & $h = 3^{-6}$&  4,750,000 & 0.02782 $\pm$ 0.00015   \\ [.1 ex]\hline
 Midpoint & $h = 3^{-4}$ & 4,650,000 & 0.02718 $\pm$ 0.00015 \\[.1 ex] \hline
 Midpoint & $h = 3^{-3}$ & 4,500,000 &  0.02537 $\pm$ 0.00015\\ [.1 ex]\hline
 Weak Trap, $\theta = 1/2$ & $h = 3^{-4}$ & 4,750,000 & 0.02840 $\pm$ 0.00015 \\[.1 ex] \hline
 Weak Trap, $\theta = 1/2$ & $h = 3^{-3}$ & 4,750,000 & 0.02838 $\pm$ 0.00015 \\[.1 ex] \hline
 Weak Trap, $\theta = 1/2$ & $h = 3^{-2}$ & 4,800,000 & 0.02946 $\pm$ 0.00015 \\ \hline
\end{tabular}.
\caption{Approximation of $P\{X_{\text{Dimer}}(1) \ge 6,000\}$ using different methods and different step sizes. As expected, the weak trapezoidal method demonstrates significantly less bias than do the Euler and midpoint methods.}
\label{table5}
\end{center}
\end{table}

\bibliographystyle{amsplain} 
\bibliography{WeakTrap.bib}

\end{document}